\documentclass[12pt]{amsart}
\usepackage{amsmath,latexsym,amsfonts,amssymb,amsthm}
\usepackage{geometry}
\usepackage{hyperref}
\usepackage{mathrsfs}
\usepackage{graphicx,color}
\usepackage{tikz-cd}
\usepackage{tikz}
\usetikzlibrary{positioning}
\usepackage{mathtools}
\usepackage[all,cmtip]{xy}

\newtheorem{prop}{Proposition}[section]
\newtheorem{thm}[prop]{Theorem}
\newtheorem{cor}[prop]{Corollary}

\newtheorem{lem}[prop]{Lemma}

\theoremstyle{definition}
\newtheorem{que}[prop]{Question}
\newtheorem{defn}[prop]{Definition}
\newtheorem{expl}[prop]{Example}
\newtheorem{rem}[prop]{\it Remark}

\newtheorem{claim}[prop]{Claim}

\newtheorem*{claim*}{Claim}

\newcommand{\bP}{\mathbb{P}}
\newcommand{\bC}{\mathbb{C}}
\newcommand{\bR}{\mathbb{R}}

\newcommand{\bQ}{\mathbb{Q}}
\newcommand{\bZ}{\mathbb{Z}}

\newcommand{\tX}{\widetilde{X}}

\newcommand{\cO}{\mathcal{O}}

\newcommand{\cI}{\mathcal{I}}

\newcommand{\cF}{\mathcal{F}}

\newcommand{\Proj}{\mathbf{Proj}}

\newcommand{\Supp}{\mathrm{Supp}~}
\newcommand{\lct}{\mathrm{lct}}
\newcommand{\mult}{\mathrm{mult}}

\newcommand{\vol}{\mathrm{vol}}

\title{Fano varieties with large Seshadri constants}
\author{Ziquan Zhuang}
\address{Department of Mathematics, Princeton University, Princeton, NJ, 08544-1000.}
\email{zzhuang@math.princeton.edu}
\date{}
\subjclass[2010]{14J45 (primary), 14E99, 14C20 (secondary)}

\begin{document}

\maketitle

\begin{abstract}
We show that the set of Fano varieties (with arbitrary singularities) whose anticanonical divisors have large Seshadri constants satisfies certain weak and birational boundedness. We also classify singular Fano varieties of dimension $n$ whose anticanonical divisors have Seshadri constants at least $n$, generalizing an earlier result of Liu and the author.
\end{abstract}


\section{Introduction}

Throughout this paper, we work over the field of complex numbers $\bC$.

Let $X$ be a projective variety and $L$ an ample line bundle on $X$. The Seshadri constants of $L$, originally introduced by Demailly \cite{demailly}, serve as a measure of the local positivity of the line bundle $L$.

\begin{defn}
Let $L$ be a nef $\bQ$-Cartier $\bQ$-divisor on a normal projective variety $X$ and $x\in X$ a smooth point. The \textit{Seshadri constant} of $L$ at $x$ is defined as
\[
  \epsilon(L,x):=\sup\{t\in\bR_{\ge 0}\mid \sigma^*L-tE\textrm{ is nef}\},
\]
where $\sigma:\mathrm{Bl}_x X\to X$ is the blow-up of $X$ at $x$, and $E$ is the exceptional divisor of $\sigma$. We also define $\epsilon(L)$ to be the maximum of $\epsilon(L,x)$ as $x$ varies over all smooth points of the variety.
\end{defn}

This invariant has many interesting properties. For example, lower bounds for Seshadri constants imply jet separation of adjoint linear series \cite{demailly}. It is also well known that the Seshadri constant of a divisor (viewed as a function on the smooth locus of the variety $X$) is lower semi-continuous and its maximum is attained at a very general point $x$ of $X$.

If $X$ is a complex Fano variety, i.e. $-K_X$ is $\bQ$-Cartier and ample, then it is natural to look at the Seshadri constant of the anticanonical divisor. While the Seshadri constants of ample divisors can {\it a priori} be arbitrarily large, this is not the case for $-K_X$ and it turns out that if $\epsilon(-K_X)$ is large then the choice of $X$ is quite restricted. Indeed, it is proved by Bauer and Szemberg \cite{bs} that if $X$ is a smooth Fano variety of dimension $n$ with $\epsilon(-K_X)>n$ then $X$ is isomorphic to the projective space (under a stronger assumption $\epsilon(-K_X)\ge n+1$ the same statement even holds in positive characteristic by Murayama \cite{takumi}). This is generalized recently to Fano varieties with klt singularities, i.e. $\bQ$-Fano varieties, by Y. Liu and the author:

\begin{thm} \cite{lz} \label{thm:P^n-weak}
Let $X$ be a Fano variety of dimension $n$ with klt singularities. Assume that $\epsilon(-K_X)>n$, then $X\cong \bP^n$.
\end{thm}

In addition, \cite{lz} classifies $\bQ$-Fano varieties with $\epsilon(-K_X)=n$. It follows from the classification that although such $\bQ$-Fano varieties do not form a bounded family, they're all rational and their anticanonical volume $\left((-K_X)^n\right)$ is bounded from above by a constant that only depends on the dimension (note that a lower bound on Seshadri constant {\it a priori} would only imply a lower bound on the volume of the divisor). In other words, the set of $\bQ$-Fano varieties with $\epsilon(-K_X)=n$ is weakly and birationally bounded.

The purpose of this article is to generalize these findings (i.e. weak/birational boundedness) to (log) Fano varieties whose (log) anticanonical divisors have ``large'' Seshadri constants and to remove the assumption on singularities in the aforementioned results. To see the optimal assumption on Seshadri constant, we look at the following example:

\begin{expl} \cite[Example 21]{lz} \label{exa:wp}
Let $X$ be the weighted projective space $\bP(1,a_1,\cdots,a_n)$ where $a_1\le\cdots\le a_n$ and $x=[1:0:\cdots:0]$, then $\mathrm{Aut}(X)\cdot x$ is Zariski open in $X$, thus
\[
\epsilon(-K_X)=\epsilon(-K_X,x)=\frac{1}{a_n}(1+a_1+\cdots+a_n)
\]
Suppose $\epsilon(-K_X)>n-1+\epsilon$ for some $\epsilon>0$, then we have $\frac{1+a_1}{a_n}>\epsilon$, hence
\[
\left((-K_X)^n\right)=\frac{(1+a_1+\cdots+a_n)^n}{a_1\cdots a_n}
\]
is bounded from above by a constant depending only on $n$ and $\epsilon$. On the other hand, let $a_1=1$, $a_2=\cdots=a_n=d$, then $\epsilon(-K_X)=n-1+\frac{2}{d}>n-1$ and
\[
\left((-K_X)^n\right)=\frac{(2+(n-1)d)^n}{d^{n-1}}\rightarrow\infty
\]
as $d\rightarrow\infty$.
\end{expl}

It follows that we can only hope for a volume upper bound assuming $\epsilon(-K_X)>n-1+\epsilon$ for some fixed $\epsilon>0$. This is exactly the content of our first result:

\begin{thm} \label{thm:volbdd}
Let $\epsilon>0$, then there exists a number $M(n,\epsilon)>0$ depending only on $n$ and $\epsilon$ with the following property: if $D$ is an effective $\bQ$-divisor on a normal projective variety $X$ of dimension $n$ such that $L=-(K_X+D)$ is nef and $\epsilon(L)>n-1+\epsilon$, then $(L^n)\le M(n,\epsilon)$.
\end{thm}

As a corollary to the above theorem, we get birational boundedness of Fano varieties $X$ with $\epsilon(-K_X)>n-1+\epsilon$. Indeed, we can be more precise:

\begin{thm} \label{thm:birbdd}
Let $D$ be an effective $\bQ$-divisor on a normal projective variety $X$ of dimension $n$ such that $L=-(K_X+D)$ is nef. If $\epsilon(L)> n-1$, then $X$ is birational to a Fano variety $Y$ with terminal singularities such that $\epsilon(-K_Y)\ge\epsilon(L)$. In particular, $X$ is rationally connected.
\end{thm}

By the well known Borisov-Alexeev-Borisov conjecture (proved by Birkar \cite{birkar}), Fano varieties with terminal singulaities form a bounded family, hence the above theorem implies the birational boundedness of Fano varieties under the weaker assumption $\epsilon(-K_X)>n-1$. Note that the condition on Seshadri constant in the theorem is again sharp, since there are cubic hypersurfaces $X$ with $\epsilon(-K_X)=n-1$, e.g. $(x_0^3+x_1^3+x_2^3=0)\subseteq\bP^{n+1}$, yet $X$ is birational to the product of an elliptic curve with $\bP^{n-1}$ and therefore is not rationally connected. Nevertheless, we still have birational boundedness in the equality case:


\begin{thm} \label{main:n-1}
Let $(X,D)$ be a pair such that $L=-(K_X+D)$ is nef and $\epsilon(L)\ge n-1$ where $n=\dim X$, then $X$ is birational to either a Fano variety $Y$ with canonical singularity such that $\epsilon(-K_Y)\ge n-1$ or the product of an elliptic curve with $\bP^{n-1}$.
\end{thm}

Theorem \ref{thm:birbdd} and \ref{main:n-1} are in fact statements about the possible outcomes of certain Minimal Model Programs (MMP), which is a powerful tool in birational geometry. Since we don't impose any assumption on the singularities of the pair in these theorems, a natural first step is to replace the pair $(X,D)$ by a terminal modification $(Y,\Delta)$. Although $Y$ is not Fano in general, we may run the MMP from $Y$ in the hope of obtaining a birational equivalence $Y\dashrightarrow Y'$ where $Y'$ is a terminal Fano variety and then anylyze the cases when this fails. To do this, we need to keep track of the information about Seshadri constants, which causes a problem. If $\phi:Y_1\dashrightarrow Y_2$ is a step in the MMP and $L$ is an ample divisor on $Y_1$, then $\phi_*L$ may not be ample on $Y_2$ and therefore $\epsilon(\phi_*L)$ is not well-defined. Indeed $\phi_*L$ is never ample if $\phi$ is a flip.

The solution to such issue is given by the moving Seshadri constants, first introduced in \cite{nakamaye}, which generalizes the notion of Seshadri constants to arbitrary line bundles. Recall that by a result of Demailly \cite[Theorem 6.4]{demailly}, the Seshadri constant of a nef and big line bundle $L$ at a general point $x\in X$ measures the asymptotic generation of jets at $x$ by sections of $L^{\otimes m}$. In other words, we have
\begin{equation} \label{eq:s}
    \epsilon(L,x)=\limsup_{m\rightarrow\infty}\frac{s(L^{\otimes m},x)}{m}
\end{equation}
where for a coherent sheaf $\cF$ on $X$, $s(\cF,x)$ is the largest integer such that the natural map
$H^{0}(\cF)\rightarrow H^{0}(\cF\otimes\cO_X/\mathfrak{m}_{x}^{s+1})$
is surjective (if $\cF$ is not generated by global sections at $x$ we put $s(\cF,x)=-1$). Now the expression on the right hand side makes sense for any line bundle $L$ or even any Weil $\bQ$-divisor, and we call it the \emph{moving Seshadri constant} of $L$ at $x$, denoted by $\epsilon_m(L,x)$ (this terminology first appears in \cite{nakamaye}; by \cite[Proposition 6.6]{elmn}, our definition agrees with the one in \cite{nakamaye} when $L$ is $\bQ$-Cartier). As before, we define $\epsilon_m(L)$ to be the maximum of $\epsilon_m(L,x)$ among all smooth points $x$.

It is natural to rephrase the above theorems in terms of moving Seshadri constants (see Theorem \ref{main:volbdd} and \ref{main:birbdd}) and except for Theorem \ref{main:n-1}, what we will actually prove are precisely these generalized versions. Using similar ideas, we also find that Fano varieties $X$ with $\epsilon(-K_X)\ge n$ automatically have klt singularities.

\begin{thm} \label{main:P^n}
Let $X$ be a normal projective variety of dimension $n$ such that $\epsilon_{m}(-K_X)>n$, then $X\cong\bP^n$.
\end{thm}

\begin{thm} \label{main:fanotype}
Let $X$ be a normal projective variety of dimension $n$ with $\epsilon_{m}(-K_X)=n$. Assume that one of the following holds:
    \begin{enumerate}
        \item $\vol(-K_X)>n^n$,
        \item $X$ is a surface, or
        \item there is an effective $\bQ$-divisor $D$ such that $L=-(K_X+D)$ is nef and $\epsilon(L)\ge n$.
    \end{enumerate}
Then $X$ is of Fano type. On the other hand, for any $\epsilon>0$ there exists $n$-dimensional Fano varieties $X$ with worse than log canonical singularities such that $\epsilon(-K_X)>n-\epsilon$.
\end{thm}

As a consequence, if $(X,D)$ is a pair such that $L=-(K_X+D)$ is nef and $\epsilon(L)>n$ where $n=\dim X$, then $X\cong\bP^n$ and a direct computation shows that $\deg D<1$. Similarly, the classification of Fano varieties $X$ with $\epsilon(-K_X)=n$ in \cite[Theorem 3]{lz} holds without any singularity assumptions.

Our next goal is to study varieties with $\epsilon_m(-K_X)=n$ in greater detail. Since moving Seshadri constant is preserved under small birational maps, we only aim to classify them up to isomorphism in codimension one. Because of Theorem \ref{main:fanotype}, it is natural to expect that all varieties $X$ with $\epsilon_m(-K_X)=n$ are of Fano type. If this is the case, by taking $Y=\Proj\oplus_{k=0}^\infty H^0(X,-kK_X)$, we obtain a birational contraction (see Definition \ref{def:bircont}) $f:X\dashrightarrow Y$ to a $\bQ$-Fano variety $Y$ such that $K_X+\Gamma=f^*K_Y$ for some effective $f$-exceptional divisor $\Gamma$. In this situation we clearly have $\epsilon_m(-K_X)=\epsilon_m(f^*(-K_Y)+\Gamma)=\epsilon_m(-K_Y)=\epsilon(-K_Y)$. In particular, $\epsilon_m(-K_X)=n$ is equivalent to $\epsilon(-K_Y)=n$ and given such $Y$ (as classified in \cite{lz}) it is not hard to list all corresponding $X$ up to small birational equivalence (see \S \ref{sec:birbdd}).

Unfortunately we don't know whether $\epsilon_m(-K_X)=n$ implies $X$ being of Fano type in general, although by Theorem \ref{main:fanotype} we are only left the boundary case $\vol(-K_X)=n^n$ (notice that $\epsilon_{m}(-K_X)=n$ implies $\vol(-K_X)\ge n^n$). Still, we have a similar yet weaker statement on the existence of birational contraction:

\begin{prop} \label{prop:contract}
Let $X$ be a projective variety of dimension $n$ with klt singularities such that $\epsilon_{m}(-K_X)=n$, then there exists a birational contraction $f:X\dashrightarrow Y$ to a $\bQ$-Fano variety $Y$ with $\epsilon(-K_Y)\ge n$.
\end{prop}

In view of this proposition, another way to classify varieties with $\epsilon_m(-K_X)=n$ is to consider varieties that admit a birational contraction $f:X\dashrightarrow Y$ such that $\epsilon(-K_Y)\ge n$ for each possible $Y$. If $\vol(-K_Y)=n^n$ this is quite straightforward (see Remark \ref{rem:vol(Y)=n^n}) and in the surface case $X$ is already of Fano type by Theorem \ref{main:fanotype}. We next carry out the corresponding analysis when $Y\cong\bP^n$. 
 
\begin{thm} \label{main:blowupPn}
Let $X$ be a variety of dimension $n\ge3$. Assume $\epsilon_{m}(-K_{X})=n$ and there exists a birational contraction
$f:X\dashrightarrow\mathbb{P}^{n}$. Then there exists a hyperplane $H\subseteq\bP^n$ such that $X$ is isomorphic in codimension one to a successive blowup of $\mathbb{P}^{n}$ along hypersurfaces in the strict transforms of $H$.
\end{thm}

In particular, the conclusion of Theorem \ref{main:fanotype} (i.e. $X$ is of Fano type) also holds in this case. However, the remaining case when $Y$ is a weighted hypersurface of degree $d+1$ in $\bP(1^{n+1},d)$ seems more complicated and it is not clear to us how to proceed.

\subsection*{Outline}

This paper is organized as follows. In \S \ref{sec:prelim} we collect some preliminary results that will be used later. In \S \ref{sec:moving} we prove some basic properties of moving Seshadri constant. The proof of Theorems \ref{thm:volbdd} and \ref{main:P^n}, based on the connectedness lemma of Koll\'ar and Shokurov and a trick for constructing isolated non-klt center, is presented in \S \ref{sec:volbdd}. \S \ref{sec:birbdd} is devoted to the proof of Theorems \ref{thm:birbdd}, \ref{main:n-1}, \ref{main:fanotype} and Proposition \ref{prop:contract}. The idea for proving Theorem \ref{thm:birbdd} is to first replace the variety by a terminal modification and then run the MMP. If it ends with a terminal Fano variety then we are done since the moving Seshadri constant does not decrease during MMP. If instead the MMP terminates with a fibertype Mori fiber space then using the assumption on Seshadri constant we prove that the fiber is a projective space while the base is a rational curve. In particular, the original variety is rational. Proposition \ref{prop:contract} is proved in a similar fashion and the proofs of Theorems \ref{main:n-1} and \ref{main:fanotype} are obtained by combining ideas from previous parts. In \S \ref{sec:izumi} we give a second proof of Theorem \ref{main:P^n} by analyzing moving Seshadri constants on varieties that admit a birational contraction to $\bP^n$. The key ingredient here is a refined version of Izumi-type inequality for divisorial valuations whose center contains a smooth point. We then proceed to find the universal optimal constant in the corresponding inequality and classify those valuations for which such constant cannot be improved. These results may be of indepedent interest and lead to the proof of Theorem \ref{main:blowupPn} in \S \ref{sec:blowup}. Finally we illustrate some examples and propose a few interesting further questions in \S \ref{sec:example}.

\subsection*{Acknowledgement}

The author would like to thank his advisor J\'anos Koll\'ar for constant support, encouragement and numerous inspiring conversations. He also wishes to thank Yuchen Liu for fruitful discussions that especially lead to the formulation of Theorem \ref{thm:volbdd} and Lue Pan for his help with the proof of Lemma \ref{lem:surface}. Finally he is grateful to the anonymous referee(s) for careful reading of the manuscript and for the numerous comments that help improve the exposition of the article.

\section{Preliminaries} \label{sec:prelim}

\subsection{Notation and conventions}

Unless otherwise specified, all varieties are assumed to be projective and normal.

A \emph{pair} $(X,D)$ consists of a variety $X$ and an effective $\bQ$-divisor $D$ on $X$ such that $K_X+D$ is $\bQ$-Cartier. If $E$ is a prime divisor over $X$, the \emph{discrepancy} of $E$ with respect to $(X,D)$ is denoted by $a(E;X,D)$. A subvariety $Z\subseteq X$ is called a \emph{non-klt center} (resp. \emph{non-lc center}) of $(X,D)$ if it is the center of a divisor $E$ over $X$ with $a(E;X,D)\le -1$ (resp. $a(E;X,D)< -1$). Similarly if the pair $(X,D)$ is canonical (see \cite[Definition 2.34]{km98} for related definitions) then $V\subseteq X$ is called a \emph{center of canonical singularity} if it's the center of a divisor $E$ with $a(E;X,D)=0$. The \emph{non-klt $($resp. non-lc$)$ locus} Nklt$(X,D)$ (resp. Nlc$(X,D)$) is the union of all non-klt (resp. non-lc) centers of $(X,D)$.

A dominant morphism $f:X\rightarrow Y$ is called a \emph{fibertype} morphism if it has connected fibers and $0<\dim Y<\dim X$.

\subsection{Minimal model program}
We will only use the minimal model program (MMP) for varieties whose canonical divisor is not pseudo-effective. In such case the existence and termination of the MMP has been established by Birkar-Cascini-Hacon-McKernan.

\begin{defn}
Let $(X,D)$ be a klt pair and $f:X\rightarrow Y$ a projective morphism with connected fibers. Then $f$ is called a \emph{Mori fiber space} if
\begin{enumerate}
    \item $X$ is $\bQ$-factorial;
    \item the relative Picard number $\rho(X/Y)=1$;
    \item $-(K_X+D)$ is $f$-ample.
\end{enumerate}
\end{defn}

\begin{thm} \cite[Corollary 1.3.3]{bchm} \label{thm:mmp}
Let $(X,D)$ be a $\bQ$-factorial klt pair. Suppose that $K_X+D$ is not pseudo-effective, then we may run a $(K_X+D)$-MMP $f:X\dashrightarrow Y$ and end with a Mori fiber space $g:Y\rightarrow Z$. 
\end{thm}

In general $X$ is not klt or even $\bQ$-Gorenstein, so we will instead run the MMP on the various modifications of $X$.

\begin{defn}
A projective birational morphism $\phi:Y\rightarrow X$ is called a \emph{terminal modification} of $X$ if $Y$ is $\bQ$-factorial, terminal and $K_Y$ is $\phi$-nef.
\end{defn}

\begin{defn}
A projective birational morphism $\phi:Y\rightarrow X$ is called a \emph{small $\bQ$-factorial modification} of $X$ if $Y$ is $\bQ$-factorial and $\phi$ is small (i.e. there is no $\phi$-exceptional divisor).
\end{defn}

For simplicity, We will just call $Y$ a terminal (resp. small $\bQ$-factorial) modification of $X$. The existence of terminal modification of a variety follows from \cite{bchm} while by \cite[Corollary 1.37]{mmp}, small $\bQ$-factorial modification of $X$ exists if there exists a divisor $\Delta$ such that $(X,\Delta)$ is dlt. We will use the following property of terminal  modification.

\begin{lem} \label{lem:discrep}
Let $(X,D)$ be a pair and $\phi:Y\rightarrow X$ a terminal modification of $X$. Then $a(E;X,D)\le 0$ for all $\phi$-exceptional divisor $E$.
\end{lem}

\begin{proof}
We may write $K_Y+\Delta\sim_\bQ \phi^*(K_X+D)$ for some divisor $\Delta$ such that $\phi_*\Delta=D$. Since $-\Delta\sim_{\phi.\bQ}K_Y$ is $\phi$-nef and $D$ is effective, $\Delta$ is also effective by the negativity lemma \cite[Lemma 3.39]{km98}. As $-a(E;X,D)$ is the coefficient of $E$ in $\Delta$, the lemma follows.
\end{proof}

\subsection{Varieties of Fano type}
\begin{defn}
A pair $(X,D)$ is called \emph{log Fano} if it is klt and $-(K_X+D)$ is ample. We say a variety $X$ is \emph{of Fano type} if there exists a divisor $D$ such that $(X,D)$ is log Fano.
\end{defn}

$\bQ$-factorial varieties of Fano type are Mori dream spaces by \cite[Corollary 1.3.2]{bchm}. In particular we can run the $D$-MMP for any divisor $D$.

\begin{defn} \label{def:bircont}
A birational map $f:X\dashrightarrow Y$ is called a \emph{birational contraction} if for a common resolution $p:W\rightarrow X$, $q:W\rightarrow Y$, every $p$-exceptional divisor is also $q$-exceptional.
\end{defn}

\begin{lem} \cite[Lemma 2.4]{bir16} \label{lem:ft1}
Let $f:X\dashrightarrow Y$ be a birational contraction. If $X$ is of Fano type, then $Y$ is also of Fano type. In particular, being of Fano type is preserved by running MMP.
\end{lem}

\begin{lem} \label{lem:ft2}
Let $(Y,\Delta)$ be a pair and $f:X\dashrightarrow Y$ a birational contraction such that $a(E;Y,\Delta)\le0$ for every $f$-exceptional divisor $E$. Assume that there exists a $\bQ$-divisor $\Delta'\ge0$ such that $(Y,\Delta+\Delta')$ is log Fano, then $X$ is of Fano type.
\end{lem}

\begin{proof}
The proof is similar to that of \cite[Lemma 2.4]{bir16}. Replace $\Delta$ by $\Delta+\Delta'$, we may assume that $\Delta'=0$. Since $H=-(K_Y+\Delta)$ is ample, we can choose a general $D\sim_\bQ H$ such that $(Y,\Delta+D)$ is still klt. Write
\[
K_X+\Delta_1+D_1+\Gamma=f^*(K_Y+\Delta+D)\sim_\bQ 0
\]
where $\Delta_1,D_1$ are strict transforms of $\Delta,D$ and $\Gamma$ is $f$-exceptional. Then $(X,\Delta_1+D_1+\Gamma)$ is also klt and by assumption $\Gamma\ge0$. As $D_1$ is big, we may write $D_1=A+E$ where $A$ is ample and $E$ is effective, then for $0<\epsilon\ll 1$ the pair $(X,\Delta_1+(1-\epsilon)D_1+\epsilon E+\Gamma)$ is log Fano, proving the lemma.
\end{proof}

\begin{cor} \label{cor:ft3}
Let $(Y,\Delta)$ be a plt pair and $f:X\dashrightarrow Y$ a birational contraction such that $a(E;Y,\Delta)\le0$ for every $f$-exceptional divisor $E$. Assume that $\lfloor \Delta \rfloor$ is $\bQ$-Cartier and $-(K_Y+\Delta)$ is ample, then $X$ is of Fano type.
\end{cor}

\begin{proof}
Let $H\subseteq Y$ be a very ample divisor containing the closure of the image of all the $f$-exceptional divisors. Choosing $H$ to be general we may assume that $H$ does not contain any component of $\lfloor \Delta \rfloor$. Since $(Y,\Delta)$ is plt, for $0<\epsilon\ll \delta\ll 1$ the pair $(Y,\Delta'=\Delta-\epsilon\lfloor \Delta \rfloor+\delta H)$ is klt. In addition, $-(K_Y+\Delta')$ is still ample and we still have $a(E;Y,\Delta')\le0$ for every $f$-exceptional divisor $E$. Hence the statement follows from Lemma \ref{lem:ft2}.
\end{proof}

\subsection{Non-klt centers}
The following results prove to be useful when dealing with non-klt pairs later.

\begin{lem} \cite[Theorem 1.1(5)]{fujino} \label{lem:raylength} 
Let $(X,D)$ be a pair and $R$ an $(K_X+D)$-negative extremal ray such that
\[R\cap \overline{NE}(X)_{\mathrm{Nlc}(X,D)}=\{0\}\]
where $\overline{NE}(X)_{\mathrm{Nlc}(X,D)}=\mathrm{Im}(\overline{NE}(\mathrm{Nlc}(X,D))\rightarrow\overline{NE}(X))$.
Then $R$ is generated by a rational curve $C$ such that $0<-(K_{X}+D\cdot C)\le2\dim X$.
\end{lem}

\begin{lem} \label{lem:connectedness}
Let $(X,D)$ be a pair such that $-(K_X+D)$ is ample. Then $\mathrm{Nklt}(X,D)$ is connected.
\end{lem}

\begin{proof}
This is a special case of the connectedness lemma \cite[17.4]{flip} when the target space is a point.
\end{proof}

\begin{lem} \label{lem:klt} 
Let $D$ be an effective $\bQ$-divisor on a smooth variety $X$ and $x\in X$. Suppose that $\mult_x D<1$, then $(X,D)$ is terminal in a neighbourhood of $x$.
\end{lem}

\begin{proof}
This follows from \cite[3.14.1]{kol95}. See also Lemma \ref{lem:terminal}.
\end{proof}

\subsection{Volume of divisors}

\begin{defn}
Let $X$ be a proper normal variety and $D$ a $\bQ$-divisor on $X$. The \emph{volume} of $D$ is defined as
\[\vol(D)=\limsup_{m\rightarrow\infty}\frac{h^0(mD)}{m^n/n!}\]
where $h^0(mD)=\dim H^0(X,\cO_X(\left\lfloor mD \right\rfloor))$.
\end{defn}

By \cite[Theorem 3.5]{mihai}, the $\limsup$ is actually a limit and this definition agrees with the usual definition of volume \cite[2.2C]{laz} when $D$ is $\bQ$-Cartier. For example, if $D$ is nef then $\vol(D)=(D^n)$.

If $f:X\dashrightarrow Y$ is a birational map and $D$ is a divisor on $Y$ then we can define the birational pullback $f^*D$ as $p_*q^*D$ where $p:W\rightarrow X$ and $q:W\rightarrow Y$ resolve the indeterminacy of $f$. It should be noted that in general $g^*f^*\neq (gf)^*$ for birational maps $f:X\dashrightarrow Y$ and $g:Y\dashrightarrow Z$.

\begin{lem} \label{lem:vol}
Let $f:X\dashrightarrow Y$ be a birational contraction, $L$ a big and nef line bundle on $Y$ and $E$ an effective divisor on $X$, then $\vol(f^*L-E)\le\vol(L)$ with equality if and only if $E=0$.
\end{lem}

\begin{proof}
Since $f$ is a birational contraction, $f_*$ induces an isomorphism $H^0(X,mf^*L)\cong H^0(Y,mL)$, hence $\vol(L)=\vol(f^*L)$. By \cite[Theorem A]{mihai}, $\vol(f^*L-E)\le\vol(f^*L)$ with equality if and only if $E\le N_\sigma(f^*L)$. As $L$ is big and nef, $f^*L$ is movable, hence $N_\sigma(f^*L)=0$ and the lemma follows.
\end{proof}

\section{The moving Seshadri constant} \label{sec:moving}

Recall that the moving Seshadri constant of a $\bQ$-divisor is defined by (\ref{eq:s}). More generally, if $L$ is a divisor on $X$ and $W$ is a sub linear system of $|L|$ then we can define $s(W,x)$ similarly as the largest integer $s$ such that $W$ generates all $s$-jets at $x$. The  definition of moving Seshadri constant then extends to sequence of linear systems as well. We leave the details to the reader.

It follows almost immediately from the definition and the lower semi-continuity of $s(\cF,x)$ that many properties
that hold true for the usual Seshadri constants generalize to moving Seshadri constants. For example, we have $\sqrt[n]{\vol(L)}\ge\epsilon_m(L)$ and $\epsilon_{m}(L)>0$ if and only if $L$ is big. It is also clear that if $D$ is an effective divisor then  $\epsilon_{m}(L)\ge\epsilon_{m}(L-D)$. Moreover if $\epsilon_{m}(L,x)\ge\lambda$
for some smooth point $x$, then $\epsilon_{m}(L,x)\ge\lambda$
for very general point $x\in X$. In particular, the moving Seshadri
constant attains its maximum $\epsilon_{m}(L)$
at very general point of $X$ (similarly we let $s(L)$ or $s(W)$ be the $s$-value of the corresponding divisor or linear system at a very general point). We can also compare moving Seshadri constants of a divisor and its restriction to a subvariety.

\begin{lem} \label{lem:restrict_m}
Let $L$ be a $\bQ$-Cartier $\bQ$-divisor on $X$. Let $Y$ be a positive dimensional subvariety of $X$ and $x$ a smooth point of both $X$ and $Y$. 
Then $\epsilon_m(L,x)\le\epsilon_m(L|_{Y},x)$.
\end{lem}

\begin{proof}
Consider the following commutative diagram
\[
\xymatrix{
H^{0}(X,L)\ar[r]\ar[d] & H^{0}(X,L\otimes\mathcal{O}_{X}/\mathfrak{m}_x^{s})\ar@{->>}[d]\\
H^{0}(Y,L|_{Y})\ar[r] & H^{0}(Y,L|_{Y}\otimes\mathcal{O}_{Y}/\mathfrak{m}_x^{s})
}
\]
If the top row is surjective, so is the bottom row. Hence $s(L,x)\le s(L|_{Y},x)$
for any Cartier divisor $L$ and $\epsilon_{m}(L,x)=\limsup\,\frac{s(mL,x)}{m}\le\limsup\,\frac{s(mL|_{Y},x)}{m}=\epsilon_{m}(L|_{Y},x)$.
\end{proof}

Apart from these similarities with the usual Seshadri constants, the
moving Seshadri constants have the additional nice property that they
never decrease under birational contraction:

\begin{lem} \label{lem:nondec}
Let $\phi:X\dashrightarrow Y$ be a birational contraction between
normal varieties and $D$ an effective divisor on $X$. Let $x\in X$
be a smooth point such that $\phi$ is an isomorphism in a neighbourhood
of $x$, then $\epsilon_{m}(D,x)\le\epsilon_{m}(\phi_{*}D,\phi(x))$.
\end{lem}

\begin{proof}
Since $\phi$ is a birational contraction, it induces an injection $\phi_*:H^0(X,\cO_X(D))\rightarrow H^0(Y,\cO_Y(\phi_*D))$ for any divisor $D$ on $X$. The result then follows by considering a similar diagram as in the previous lemma:
\[
\xymatrix{
H^{0}(X,\cO_X(D))\ar[r]\ar[d] & H^{0}(X,\cO_X(D)\otimes\mathcal{O}_{X}/\mathfrak{m}_x^{s})\ar[d]^\cong\\
H^{0}(Y,\cO_Y(\phi_*D))\ar[r] & H^{0}(Y,\cO_Y(\phi_*D)\otimes\mathcal{O}_{Y}/\mathfrak{m}_{\phi(x)}^{s})
}
\]
\end{proof}

\begin{cor} \label{cor:nondecrease}
Let $\phi:X\dashrightarrow Y$ be a birational
contraction between normal varieties, then $\epsilon_{m}(-K_{X})\le\epsilon_{m}(-K_{Y})$.
\end{cor}

In addition, the moving Seshadri constant of anticanonical divisor is preserved by taking terminal modification:

\begin{lem} \label{lem:mmodel}
Let $\phi:Y\rightarrow X$ be a terminal modification of $X$, then $\epsilon_m(-K_X)=\epsilon_m(-K_Y)$.
\end{lem}

\begin{proof}
Let $D_X\in|-mK_X|$ and $D_Y$ its strict transform on $Y$. We may write $mK_Y+D_Y+E_Y\sim\phi^*(mK_X+D_X)\sim 0$ for some $\phi$-exceptional divisor $E_Y$. Note that $E_Y$ has integral coefficients. Apply Lemma \ref{lem:discrep} to the pair $(X,\frac{1}{m}D_X)$ we see that $E_Y$ is effective. It follows that $D_Y+E_Y\in|-mK_Y|$ and we have an injection $\phi^{-1}_*:H^0(X,-mK_X)\rightarrow H^0(Y,-mK_Y)$. On the other hand $\phi_*$ also induces an inclusion $H^0(Y,-mK_Y)\rightarrow H^0(X,-mK_X)$ and $\phi_*\circ\phi^{-1}_*=\mathrm{id}$, hence it's an isomorphism and the lemma follows.
\end{proof}

To give upper bounds of moving Seshadri constant we will usually use the following observation.

\begin{lem} \label{lem:sbound}
Let $L$ be a $\bQ$-Cartier divisor on $X$ and $W\subseteq|L|$ a sub linear system. Let $x$ be a smooth point on $X$ and $C$ an irreducible curve containing $x$. Assume that $x$ is not a base point of $W$. Then
\[s(W,x)\le \frac{(L\cdot C)}{\mult_x C}\]
and the inequality is strict if $C$ intersects the base locus of $W$.
\end{lem}

\begin{proof}
Since $x\not\in\mathrm{Bs}(W)$, we have $s=s(W,x)\ge 0$. As $W$ generates $s$-jets at $x$, we may choose $D\in W\subseteq |L|$ such that $C\not\subseteq D$ and $\mult_x D=s$. It then follows that
\[(L\cdot C)\ge\mult_x C\cdot \mult_x D=s\cdot \mult_x C\]
and the inequality is strict if $C$ intersects $D$ at points other than $x$. In particular, this happens if $C$ intersects the base locus of $W$.
\end{proof}

We now rephrase Theorem \ref{thm:volbdd} and \ref{thm:birbdd} using moving Seshadri constant. Theorem \ref{thm:volbdd} and \ref{thm:birbdd} will then follow immediately as special cases of these more general versions.

\begin{thm} \label{main:volbdd}
Let $\epsilon>0$, then there exists a number $M(n,\epsilon)>0$ depending only on $n$ and $\epsilon$ with the following property: if $X$ is a normal projective variety of dimension $n$ such that $\epsilon_m(-K_X)>n-1+\epsilon$, then $\vol(-K_{X})\le M(n,\epsilon)$.
\end{thm}

\begin{thm} \label{main:birbdd}
Let $X$ be a normal projective variety of dimension $n$ such that $\epsilon_m(-K_X)> n-1$, then  $X$ is birational to a terminal Fano variety $Y$ with $\epsilon(-K_Y)\ge\epsilon_m(-K_X)$.
\end{thm}

We will prove these statements in subsequent sections.

\section{Weak boundedness} \label{sec:volbdd}

We start with the proof of Theorem \ref{main:volbdd}.

\begin{proof}[Proof of Theorem \ref{main:volbdd}]
We may assume $\epsilon\in\mathbb{Q}$. Since $\epsilon_m(-K_X)>0$, $-K_X$ is big, hence we may write $-K_X\sim_\bQ A+E$ where $A$ is ample and $E$ is effective. Choose $a,b,c>0$ such that $(n-1+\epsilon)a\ge n-1+\frac{\epsilon}{2}$
and $a+b+c<1$. Suppose $\vol(-K_X)>\max\{b^{-n}(1-\frac{\epsilon}{2})^{n},c^{-n}n^{n}\}$.
Let $y$ be a smooth point of $X$. By \cite[Lemma 10.4.11]{laz2}, there exists $D_{3}\sim_{\mathbb{Q}}-cK_X$
such that $\mathrm{mult}_{y}D_{3}\ge n$. Let $x$ be very general
point of $X$ (so $x$ is a smooth point and is not contained in the
support of $E$ and $D_{3}$). By \cite[Lemma 10.4.11]{laz2} again there exists $D_{2}\sim_{\mathbb{Q}}-bK_X$
such that $\mathrm{mult}_{x}D_{2}=1-\frac{\epsilon}{2}$. By the upper
semi-continuity of multiplicities, we have $\mathrm{mult}_{p}(D_{2})\le1-\frac{\epsilon}{2}$
when $p$ varies in a Zariski neighbourhood of $x$. Since $\epsilon_m(-K_X,x)>n-1+\epsilon$,
the linear system $|-mK_X|$ generates $m(n-1+\epsilon)$-jets at $x$ for $m\gg0$. Hence we can choose $H\in |-mK_X|$ ($m\gg0$) such that $H$ has an isolated singularity of multiplicity $m(n-1+\epsilon)$ at $x$. Let $D_{1}=\frac{a}{m}H$, then we have
$\mathrm{mult}_{x}D_{1}\ge n-1+\frac{\epsilon}{2}$ and $\mathrm{mult}_{p}D_{1}\le\frac{a}{m}\ll 1$
for all smooth point $p$ in a punctured neighbourhood of $x$. Let $\Delta=D_{1}+D_{2}+D_{3}+(1-a-b-c)E$. By
construction we have $\mathrm{mult}_{x}\Delta\ge n$, $\mathrm{mult}_{y}\Delta\ge n$
and $\mathrm{mult}_{p}\Delta<1$ when $p$ lies in a punctured neighbourhood
of $x$. It follows that $x,y\in\mathrm{Nklt}(X,\Delta)$ and by Lemma \ref{lem:klt}, $x$ is an isolated point in $\mathrm{Nklt}(X,\Delta)$. In particular, $\mathrm{Nklt}(X,\Delta)$ is not connected. But as $-(K_{X}+\Delta)\sim_{\mathbb{Q}}(1-a-b-c)A$
is ample, this contradicts Lemma \ref{lem:connectedness}. Hence we must have $\vol(-K_X)\le\max\{b^{-n}(1-\frac{\epsilon}{2})^{n},c^{-n}n^{n}\}$.
\end{proof}

It is not hard to see from the proof that we can take $M(n,\epsilon)=O\left(\frac{n^{2n}}{\epsilon^n}\right)$. We also note the following consequence of the same proof method.


\begin{cor} \label{cor:ft}
Let $X$ be a normal projective variety of dimension $n$ such that
$\epsilon_{m}(-K_{X})>n$ or $\epsilon_{m}(-K_{X})=n$ and $\mathrm{Vol}(-K_{X})>n^{n}$,
then $X$ is of Fano type.
\end{cor}

\begin{proof}
Let $x$ be a very general point of $X$ and write $-K_X\sim_\bQ A+E$ as before. If $\epsilon_m(-K_X)>n$ then we can choose $D\sim_\bQ -K_X$ such that $\mult_x D > n$ while $\mult_p D \ll 1$ when $p$ lies in a punctured neighbourhood of $x$. Let $a=\lct_x(X,D)<1$, $\Delta=aD+(1-a)E$, then $(X,\Delta)$ is lc and has an isolated non-klt center at $x$. Since $-(K_X+\Delta)\sim_\bQ (1-a)A$ is ample, by Lemma \ref{lem:connectedness} we have $\mathrm{Nklt}(X,\Delta)=\{x\}$. Let $\Gamma=(a-t)D+(1-a+t)E$ where $0<t\ll 1$, then $(X,\Gamma)$ is log Fano.

Similarly, if $\epsilon_m(-K_X)=n$ then for any $t>0$ we can choose $D_1\sim_\bQ -K_X$ such that $\mult_x D_1 > n-t$ while $\mult_p D \ll 1$ when $p$ lies in a punctured neighbourhood of $x$. Since $\mathrm{Vol}(-K_{X})>n^{n}$, we can also choose $D_2\sim_\bQ -K_X$ such that
$\mult_x D_2 = n(1+t)$ for sufficiently small $t$. Let $D=\frac{n}{n+1}D_1+\frac{1}{n+1}D_2$, then $D\sim_\bQ -K_X$, $\mult_x D > n$ and $(X,D)$ is klt in a punctured neighbourhood of $x$. The proof now proceeds as in the previous case.
\end{proof}

We expect the conclusion of the corollary to hold under the weaker assumption that  $\epsilon_m(-K_X)=n$. Nevertheless, the above version is enough for proving Theorem \ref{main:P^n}.

\begin{proof}[Proof of Theorem \ref{main:P^n}]
By Corollary \ref{cor:ft}, $X$ is of Fano type. In particular, there exists a divisor $\Delta$ such that $(X,\Delta)$ is klt, hence there exists a small $\bQ$-factorial modification $Y$ of $X$. It suffices to show that $Y\cong\bP^n$. We may thus replace $X$ by $Y$ and assume that $X$ is $\bQ$-factorial.

Since $X$ is of Fano type, we can run the $(-K_X)$-MMP $f:X\dashrightarrow Y$ and terminates with a $\bQ$-Fano variety $Y$. By Corollary \ref{cor:nondecrease}, $\epsilon(-K_Y)=\epsilon_m(-K_Y)\ge\epsilon_m(-K_X)>n$, hence by Theorem \ref{thm:P^n-weak}, $Y\cong\bP^n$. As $f$ comes from a $(-K_X)$-MMP, we have $K_X+D=f^*K_Y$ for some effective divisor $D$ supported in the $f$-exceptional locus. Since $Y$ is smooth and in particular terminal, this is impossible unless $f$ is small, but then $X\cong\bP^n$ as well.
\end{proof}

Another consequence of Theorem \ref{main:volbdd} is the birational boundedness of varieties $X$ with $\epsilon_m(-K_X)>n-1+\epsilon$. While this is also implied by Theorem \ref{main:birbdd}, the following proof is shorter and does not use the solution of BAB conjecture.

\begin{cor}
Let $\epsilon>0$, then the set of varieties $X$ with $\epsilon_m(-K_X)>n-1+\epsilon$ is birationally bounded.
\end{cor}

\begin{proof}
By the same method as in the proof of Corollary \ref{cor:ft}, for any very general point $x\in X$ we can find $\Delta\sim_\bQ -aK_X$ such that $0<a<2$ and $(X,\Delta)$ is lc and has an isolated non-klt center at $x$. Apply \cite[Lemma 2.3.4]{hmx} to $D=-4K_X$ we see that $|-3K_X|$ defines a birational map. Since $\vol(-K_X)$ has a uniform upper bound $M(n,\epsilon)$, the set of such $X$ is birationally bounded by \cite[Lemma 2.4.2]{hmx}.
\end{proof}

\section{Birational boundedness} \label{sec:birbdd}

We now proceed to the proof of Theorem \ref{main:birbdd}. Recall that the base ideal of a line bundle $L$ on $X$, denoted by $\mathfrak{b}(L)$, is the image of the natural evaluation map $H^{0}(X,L)\otimes L^{*}\rightarrow\mathcal{O}_{X}$.  

\begin{lem} \label{lem:mult}
Let $f:X\rightarrow Y$ be a fibertype morphism with general fiber $F\cong\bP^{r-1}$. Assume that $\epsilon_{m}(-K_{X})>r-\epsilon$ for some $\epsilon>0$. Then 
$\mathrm{mult}_{x}\mathfrak{b}(-mK_{X})<m\epsilon$ for all $x\in F$ and sufficiently large and divisible $m$.

\end{lem}

\begin{proof}
By the definition of moving Seshadri constant, we have $s(-mK_X)>m(r-\epsilon)$ for sufficiently divisible $m\gg0$. Let $W_m$ be the image of the restriction map $H^0(-mK_X)\rightarrow H^0(-mK_F)$. Let $x\in F$ and $k=\mathrm{mult}_{x}\mathfrak{b}(-mK_{X})$. By the proof of Lemma \ref{lem:restrict_m} we have $s(W_m)\ge s(-mK_X)>m(r-\epsilon)$. On the other hand, we have $W_m\subseteq H^0(\mathfrak{m}_x^k\otimes \cO(-mK_F))$, hence if $\pi:\hat{F}\rightarrow F$ is the blowup of $F$ at $x$ with exceptional divisor $E$, then $W_m$ can be considered as a sub linear system of $L=\pi^{*}(-mK_{F})-kE$. Apply Lemma \ref{lem:sbound} to $L$ and the strict transform $C$ of a line joining $x$ and a very general point in $F$ we get $s(W_m)\le s(L)\le (L\cdot C)=mr-k$. It follows that $mr-k>m(r-\epsilon)$ and $k<m\epsilon$, proving the lemma.
\end{proof}

\begin{cor} \label{cor:kltxd}
With the same assumption as in the Lemma \ref{lem:mult}, there exists $D\sim_\bQ -\frac{1}{\epsilon}K_X$ such that $(X,D)$ is klt along $F$.
\end{cor}

\begin{proof}
Choose $0<\epsilon_0<\epsilon$ such that $\epsilon_{m}(-K_{X})>r-\epsilon_0$ still holds. By Lemma \ref{lem:mult}, for all sufficiently divisible $m\gg0$ and $x\in F$ we may find an effective divisor $D_0\sim-mK_X$ such that $\mult_x D_0<m\epsilon_0$. By \cite[Theorem 4.1]{kol95}, we can indeed choose $D_0$ such that $\mult_x D_0<m\epsilon_0+1<m\epsilon$ for all $x\in F$. Let $D=\frac{1}{m\epsilon}D_0$ for $m\gg0$ we see that $D\sim_\bQ-\frac{1}{\epsilon}K_X$ and $\mult_x D<1$ for all $x\in F$, thus the pair is klt along $F$ by Lemma \ref{lem:klt}.
\end{proof}

\begin{lem} \label{lem:f&b}
Let $f:X\rightarrow Y$ be a fibertype morphism with general fiber $F$ and $\Gamma$ an $f$-ample $\bQ$-divisor. Assume that $\epsilon_{m}(-K_{X}-\Gamma)\ge n-1$ where $n=\dim X$, then $F\cong\mathbb{P}^{n-1}$, and either $Y\cong\mathbb{P}^{1}$ or $\epsilon_{m}(-K_{X})=n-1$ and $Y$ is an elliptic curve. 
\end{lem}

\begin{proof}
By assumption and Lemma \ref{lem:restrict_m} we have $\epsilon_m(-K_{F})>\epsilon_m(-K_{F}-\Gamma|_{F})\ge\epsilon_{m}(-K_{X}-\Gamma)\ge n-1$
hence $F\cong\mathbb{P}^{n-1}$ by Theorem \ref{main:P^n} and $Y$ is a curve. Suppose first that $\epsilon_{m}(-K_{X})>n-1$. By Corollary \ref{cor:kltxd}, there exists $D\sim_{\mathbb{Q}}-K_{X}$ such that $(X,D)$ is klt along $F$. As $-K_{X}$ is big, there also exists $a>0$ and an effective $\mathbb{Q}$-divisor $\Delta$ such that $-K_{X}\sim_{\mathbb{Q}}aF+\Delta$. For $0<\lambda\ll1$ the pair $(X,D_{\lambda}=(1-\lambda)D+\lambda\Delta)$ is also klt along $F$ and we have $K_{X}+D_{\lambda}\sim_{\mathbb{Q}}f^{*}(-a\lambda P)$ where $P$ is a divisor of degree 1 on $Y$, hence by the canonical bundle formula \cite[Theorem 8.5.1]{Kformula} we get $-a\lambda P\sim_{\mathbb{Q}}K_{Y}+J+B$ where the moduli part $J$ and the boundary part $B$ are both pseudo-effective since $D_{\lambda}\ge0$. It follows that $\deg K_{Y}\le-a\lambda<0$, thus $Y\cong\mathbb{P}^{1}$.

Next assume $\epsilon_{m}(-K_{X})=n-1$. By Corollary \ref{cor:kltxd} again, for any $\lambda<1$ there exists $D\sim_{\mathbb{Q}}-\lambda K_{X}$ such that $(X,D)$ is klt along $F$. Since $-K_{X}|_F$ is ample, we may fix an ample divisor $H$ on $Y$ such that $r(-K_{X}+f^{*}H)$ is globally generated in a neighbourhood of $F$ for some $r\in \bZ_{>0}$. It follows that there exists $A\sim_{\mathbb{Q}}-K_{X}+f^{*}H$ such that $(X,D+(1-\lambda)A)$ is still klt along $F$. Note that $K_{X}+D+(1-\lambda)A\sim_{\mathbb{Q}}(1-\lambda)f^{*}H$, by the canonical bundle formula again we get $(1-\lambda)H\sim_{\mathbb{Q}}K_{Y}+J+B$
where $J$ and $B$ are both pseudo-effective. It follows that $\deg K_{Y}\le(1-\lambda)\deg H$. Letting $\lambda\rightarrow1$ we obtain $\deg K_{Y}\le0$, hence $Y$ is either a rational or an elliptic curve.
\end{proof}

We are ready to prove Theorem \ref{main:birbdd}.

\begin{proof}[Proof of Theorem \ref{main:birbdd}]
By Lemma \ref{lem:mmodel}, we may replace $X$ by its terminal modification and assume that it has terminal singularities. Since $\epsilon_m(-K_X)>0$, $-K_X$ is big, hence by Theorem \ref{thm:mmp} we may run the $K_X$-MMP $X\dashrightarrow X'$ and end with a Mori fiber space $f:X'\rightarrow Z$ such that $X'$ has terminal singularities and $\epsilon_m(-K_{X'})\ge\epsilon_m(-K_X)>n-1$. If $Z$ is a point then $X'$ is terminal Fano and we may just take $Y=X'$. If $Z$ is not a point then by Lemma \ref{lem:f&b} the general fiber of $f$ is isomorphic to $\bP^{n-1}$ and $Z\cong\bP^1$. Since the function field of a complex curve has trivial Brauer group, $X'$ is rational, thus by Theorem \ref{main:P^n}, $Y=\bP^n$ satisfies the conclusion of the theorem.
\end{proof}

With similar ideas, we also give the proof of Proposition \ref{prop:contract}, Theorem \ref{main:n-1} and \ref{main:fanotype}.

\begin{lem} \label{lem:mfs=n}
Let $f:X\rightarrow Y$ be a fibertype Mori fiber space with general fiber $F$. Assume that $Y$ is a curve and $\vol(-K_F)=r^{n-1}$ where $r=\epsilon_m(-K_X)$ and $n=\dim X$. Then $-K_X$ is nef and big.
\end{lem}

\begin{proof}
Let $\pi:\tilde{F}\rightarrow F$ be a resolution of singularity and let $W_m$ be the image of the restriction map $H^0(-mK_X)\rightarrow H^0(-mK_F)$. Let $\epsilon>0$, then by the proof of Lemma \ref{lem:restrict_m} we have $s(W_m)\ge s(-mK_X)>m(r-\epsilon)$ as $m\gg 0$, hence the restricted volume (in the sense of \cite{elmn}) satisfies $\vol_{X|F}(-K_X)\ge (r-\epsilon)^{n-1}$. Letting $\epsilon\rightarrow0$ we obtain $\vol_{X|F}(-K_X)=r^{n-1}=\vol(-K_F)$. It follows that for any $x\in \tilde{F}$ and $\epsilon>0$, there exists $D\sim_\bQ -K_X$ such that $\mult_x \pi^*D<\epsilon$, since otherwise $\pi^*W_m$ is contained in the sub linear system $H^0(\mathfrak{m}_x^{m\epsilon}\otimes\pi^*(-mK_F))$ and thus
\[\vol_{X|F}(-K_X)\le\vol(\mathfrak{m}_x^\epsilon\otimes(-\pi^*K_F))<\vol(-K_F)\]
by \cite[Theorem A]{mihai}, a contradiction. Letting $\epsilon\rightarrow0$, we see as in Corollary \ref{cor:kltxd} that for any $m\gg0$ there exists $D\sim_\bQ -mK_X$ such that $(X,D)$ is klt along $F$. 

By assumption $\rho(X)=\rho(Y)+\rho(X/Y)=2$, hence the Mori cone $\overline{NE}(X)$ is 2-dimensional and generated by a curve $l$ in $F$ and another extremal ray $R$ such that $(F\cdot R)>0$. Suppose that $-K_X$ is not nef. Then we have $(K_X+D\cdot R)=(m-1)(-K_X\cdot R)<0$. On the other hand, as $(X,D)$ is klt along $F$, $\mathrm{Nlc(X,D)}$ doesn't dominate $Y$ and we have $\overline{NE}(X)_{\mathrm{Nlc}(X,D)}\subseteq\bR_+[l]\subseteq\overline{NE}(X)_{K_X+D\ge0}$. By Lemma \ref{lem:raylength}, $R$ is generated by a rational curve $C$ such that $0<-(K_X+D\cdot C)\le 2n$. But $-(K_X+D\cdot C)=(m-1)(K_X\cdot C)\rightarrow\infty$ as $m\rightarrow\infty$ and we derive a contradiction. Hence $-K_X$ is nef. It is also big since $\epsilon_m(-K_X)>0$.
\end{proof}

\begin{proof}[Proof of Proposition \ref{prop:contract}]
Since $X$ is klt and $-K_X$ is big by assumption, we may run the $K_X$-MMP $f:X\dashrightarrow X'$ by Theorem \ref{thm:mmp} where $X'$ admits a Mori fiber space structure $g:X'\rightarrow Z$. By Corollary \ref{cor:nondecrease}, $\epsilon_m(-K_{X'})\ge \epsilon_m(-K_X)=n$. If $Z$ is a point then we just take $Y=X'$. If $\dim Z>0$, let $F$ be the general fiber of $g$, then we have $\epsilon(-K_F)\ge \epsilon_m(-K_{X'})\ge n$ by Lemma \ref{lem:restrict_m} thus $F\cong\bP^{n-1}$ by Theorem \ref{main:P^n} and $Z$ is a curve. But then $\vol(-K_F)=n^{n-1}$, so by Lemma \ref{lem:mfs=n}, $-K_{X'}$ is nef and big. Since $X'$ has klt singularities, $-K_{X'}$ is semiample by \cite[Theorem 3.3]{km98} and defines a birational morphism $h:X'\rightarrow Y$ such that $-K_Y$ is ample. As $h$ is crepant, $Y$ is klt and $\epsilon(-K_Y)=\epsilon(-K_{X'})\ge n$.
\end{proof}

\begin{rem} \label{rem:vol(Y)=n^n}
If the $\bQ$-Fano variety $Y$ in Proposition \ref{prop:contract} satisfies $\vol(-K_Y)=n^n$ and we write $K_X+\Gamma=f^*K_Y$, then we must have $\Gamma\ge0$. Indeed if $\Gamma=\Gamma_1-\Gamma_2$ where $\Gamma_1,\Gamma_2\ge0$ have no common components then $H^0(-mK_X)=H^0(m(-f^*K_Y+\Gamma))=H^0(m(-f^*K_Y-\Gamma_2))$, thus since $\epsilon_m(-K_X)=n$ we have $\vol(-f^*K_Y-\Gamma_2)=\vol(-K_X)\ge n^n=\vol(-K_Y)$. By Lemma \ref{lem:vol}, $\Gamma_2=0$ and hence $\Gamma\ge0$.
\end{rem}

\begin{lem} \label{lem:pair=n}
Let $(X,D)$ be a pair such that $L=-(K_X+D)$ is nef and $\epsilon(L)\ge n=\dim X$, then $(X,D)$ is klt unless $(X,D)$ is a crepant modification of $(\bP^n,H)$ $($i.e. there exists a birational contraction $f:X\dashrightarrow \bP^n$ such that $D=f^{-1}_* H$ and $K_X+D=f^*(K_{\bP^n}+H))$ where $H$ is a hyperplane in $\bP^n$. In particular, $(X,D)$ is lc and $X$ is of Fano type.
\end{lem}

\begin{proof}
By \cite[Corollary 1.4.4]{bchm} there is a birational morphism $\pi:Y\rightarrow X$ such that $Y$ is $\bQ$-factorial, $K_Y+\Gamma_1+\Gamma_2=\pi^*(K_X+D)$ where $(Y,\Gamma_1)$ is klt, every component of $\Gamma_2$ has coefficient at least one and no component of $\Gamma_1$ is exceptional. We may assume that $\Gamma_2\neq0$, otherwise there is nothing to prove. As $L$ is nef and big, we may choose an effective $M\sim_\bQ \pi^*L$ such that $(Y,\Gamma_1+M)$ is still klt. Note that $K_Y+\Gamma_1+M\sim_\bQ -\Gamma_2$ is not pesudo-effective, by Theorem \ref{thm:mmp} we may run a $(K_Y+\Gamma_1+M)$-MMP $f:Y\dashrightarrow Z$ which terminates with a Mori fiber space $g:Z\rightarrow W$ such that $f_*\Gamma_2$ is $g$-ample. Let $F$ be the general fiber of $g$, $\Gamma=\Gamma_1+\Gamma_2$ and let $G=f_*\Gamma$. By Lemma \ref{lem:nondec} we have
\[
\epsilon_m(-K_Z-G)\ge\epsilon_m(-K_Y-\Gamma)=\epsilon(-K_X-D)\ge n
\]
If $W$ is not a point then since $\rho(Z/W)=1$, $G$ is also $g$-ample and we have $\epsilon_m(-K_Z-G)\le\epsilon_m(-K_F-G|_F)<\epsilon_m(-K_F)\le n$ by Lemma \ref{lem:restrict_m} and Theorem \ref{main:P^n}, a contradiction. Hence $W$ is a point and $Z$ is a $\bQ$-Fano variety with $\epsilon(-K_Z)>\epsilon(-K_Z-G)\ge n$, so $Z\cong\bP^n$ and $\deg G\le 1$. But as $G\ge f_*\Gamma_2$ contains at least one component with coefficient at least one, this is only possible when $G$ is a hyperplane in $\bP^n$.

Write $K_Y+\Gamma=f^*(K_Z+G)+E$. Since $(Z,G)$ clearly has canonical singularities, $E\ge0$. Suppose $E>0$, then by Lemma \ref{lem:vol} we have $\vol(-K_Y-\Gamma)<\vol(-K_Z-G)=n^n$. On the other hand, $\epsilon_m(-K_Y-\Gamma)\ge n$ implies $\vol(-K_Y-\Gamma)\ge n^n$, a contradiction. It follows that $E=0$ and as the pair $(Z,G)$ is canonical, we have $\Gamma=f^{-1}_*G$. In other words, $(Y,\Gamma)$ is a crepant modification of $(Z,G)$. 

As $E=0$ and $(Z,G)$ is log Fano, $Y$ is of Fano type by Corollary \ref{cor:ft3}. By Lemma \ref{lem:ft1}, $X$ is also of Fano type. Hence if $\pi_*\Gamma=0$ then $D=0$ and $X$ has klt singularities. Otherwise $D=\pi_*\Gamma\neq0$ is irreducible and has coefficient one. Run the $(-D)$-MMP $\phi:X\dashrightarrow X'$ on $X$. The same argument as before then shows that $\phi_*D$ is a hyperplane in $X'\cong\bP^n$ and $(X,D)$ is a crepant modification of $(X',\phi_*D)$.
\end{proof}

\begin{proof}[Proof of Theorem \ref{main:n-1}]
Suppose first that $(X,D)$ is not klt. As in the proof of Lemma \ref{lem:pair=n}, let $\pi:Y\rightarrow X$ be a birational morphism such that $Y$ is $\bQ$-factorial, $K_Y+\Gamma_1+\Gamma_2=\pi^*(K_X+D)$ where $(Y,\Gamma_1)$ is klt, every component of $\Gamma_2$ has coefficient at least one and no component of $\Gamma_1$ is exceptional. Recall that $L=-(K_X+D)$. Run the $(K_Y+\Gamma_1+\pi^*L)$-MMP $f:Y\dashrightarrow Z$ where $g:Z\rightarrow W$ is a Mori fiber space such that $G=f_*\Gamma_2$ is $g$-ample. We have $\epsilon_m(-K_Z-G)\ge\epsilon_m(-K_Y-\Gamma_2)\ge\epsilon(-K_X-D)=n-1$.  If $W$ is not a point then $Z$ is either rational or birational to the product of an elliptic curve with $\bP^{n-1}$ by Lemma \ref{lem:f&b}. Hence if $X$ is not birational to $E\times\bP^{n-1}$ where $E$ is an elliptic curve then either $(X,D)$ is klt or $X$ is birational to a $\bQ$-Fano variety $Z$ with $\epsilon(-K_Z)\ge n-1$.

We may therefore assume that $X$ is of Fano type and $\epsilon_m(-K_X)\ge n-1$. Replacing $X$ by a small $\bQ$-factorial modification and then running the $(-K_X)$-MMP we may even assume that $X$ is $\bQ$-Fano. If $X$ does not have canonical singularities, by \cite[Corollary 1.39]{mmp} there exists a birational morphism $\pi:Y\rightarrow X$ (hopefully our repeated use of the same letters does not cause any confusion) with a single exceptional divisor $E$ such that $Y$ is $\bQ$-factorial and $-1<a(E;X,0)<0$. Let $a=-a(E;X,0)$, then $f^*K_X=K_Y+aE$. By Lemma \ref{lem:ft2}, $Y$ is also of Fano type and we may run the $(-E)$-MMP $f:Y\dashrightarrow Z$ where $g:Z\rightarrow W$ is a Mori fiber space such that $G=f_*E$ is $g$-ample. By Lemma \ref{lem:nondec}, $\epsilon_m(-K_Z-aG)\ge \epsilon_m(-K_Y-aE)=\epsilon(-K_X)=n-1$. If $W$ is not a point then Lemma \ref{lem:f&b} and the fact that $X$ is $\bQ$-Fano implies $X$ is rational. If $W$ is a point then $Z$ is $\bQ$-Fano and as $G$ is ample we have $\epsilon(-K_Z)>\epsilon(-K_Z-aG)\ge n-1$, thus by Theorem \ref{main:birbdd}, $Z$ (and hence $X$) is birational to a terminal Fano variety $Z'$ with $\epsilon(-K_{Z'})\ge n-1$. The proof is now complete.
\end{proof}

\begin{proof}[Proof of Theorem \ref{main:fanotype}]
We first show that under any of the assumptions (1)-(3) $X$ is of Fano type. If $\vol(-K_X)>n^n$ this follows from Corollary \ref{cor:ft} and in case (3) this is given by Lemma \ref{lem:pair=n}. Suppose that $X$ is a surface. By Lemma \ref{lem:ft1} and \ref{lem:mmodel} we may replace $X$ by its minimal resolution and assume that $X$ is smooth. Let $-K_X=P+N$ be the Zariski decomposition of $-K_X$ where $P$ is nef and $N\ge 0$ is the negative part. We have $H^0(X,\lfloor mP \rfloor)=H^0(X,-mK_X)$ for all $m\ge 0$, thus $\epsilon(P)=\epsilon_m(P)=\epsilon_m(-K_X)=2$. Apply Lemma \ref{lem:pair=n} to the pair $(X,N)$ we see that $X$ is also of Fano type.


Next we construct examples of Fano varieties $X$ with non-lc singularities such that $\epsilon(-K_X)>n-\epsilon$. Let $Y=\bP(1,a_1,\cdots,a_n)$ where $a_1\le\cdots\le a_n$ and $H=(x_0=0)$. Then $(Y,2H)$ is not lc, $L=-(K_Y+2H)$ is ample and by Example \ref{exa:wp}, $\epsilon(L)=\frac{1}{a_n}(\sum_{i=1}^n a_i -1)>n-\epsilon$ for suitable choice of $a_1, \cdots, a_n$. Now assume that the $a_i$'s are pairwise relatively prime so that $Y$ has only isolated singularities. Let $f(x_{0},\cdots,x_{n})$ be a general weighted homogeneous polynomial of degree $d\gg0$ such that $(f=0)$ is contained in the smooth locus of $Y$. Let $\pi:\tX\rightarrow Y$ be the blowup of the subscheme $Z=(x_{0}^{2}=f=0)$ with exceptional divisor $E$ and let $H$ also denote its strict transform on $\tX$. We have $-K_{\tX}\sim 2H+\pi^{*}L$, $H\cong\mathbb{P}(a_{1},\cdots,a_{n})$, $H|_{H}=(1-\frac{d}{2})M$ where $M$ is the ample generator of $\mathrm{Cl}(H)$ and $-K_{\tX}|_{H}=(1+\sum_{i=1}^{n}a_{i}-d)M$. Hence if $d-(1+\sum_{i=1}^{n}a_{i})>\frac{d}{2}-1>0$ (which is satisfied as $d\gg0$) and $\tX\rightarrow X$ is the contraction of $H$ then $X$ is Fano but not lc and $\epsilon(-K_X)\ge\epsilon_m(-K_{\tX})\ge\epsilon_m(\pi^*L)>n-\epsilon$.
\end{proof}

If $f:X\dashrightarrow Y$ is a birational contraction such that $a(E;X,0)\le0$ for all $f$-exceptional divisor $E$ then we call $f$ a partial terminal modification of $Y$. As we discuss in the introduction, if $X$ is of Fano type (e.g. when any of the assumptions in Theorem \ref{main:fanotype} holds), $\epsilon_m(-K_X)=n$ if and only if $X$ is a partial terminal modification of a $\bQ$-Fano variety $Y$ with $\epsilon(-K_Y)=n$. Using the classification of $Y$ in \cite{lz}, we list all corresponding $X$ up to small birational equivalence as follows:

\begin{enumerate}
\item Let $Y$ be a degree $d+1$ weighted hypersurface $(x_0x_{n+1}=f(x_1,\cdots,x_n))\subset\mathbb{P}(1^{n+1},d)$. Let $\pi:Y_1\rightarrow Y$ be the blowup of $p=[0:\cdots:0:1]$ and $\Gamma$ the exceptional divisor. It is not hard to see that $Y_1$ is isomorphic to the blowup of $\bP^n$ along a hypersurface $W$ of degree $d+1$ in a hyperplane $H$, $\Gamma$ is the strict transform of $H$ and $\pi^*K_Y=K_{Y_1}+(1-\frac{n}{d})\Gamma$. In particular, $Y_1$ has cA-type singularities and is smooth along $\Gamma$.  Hence the pair $(Y_1,(1-\frac{n}{d})\Gamma)$ is terminal along $\Gamma$ and is canonical away from $\Gamma$ by \cite[1.42]{mmp}. It follows that if $E$ is an exceptional divisor over $Y$ such that $a(E;Y,0)\le0$ then either $E=\Gamma$ or $a(E;Y_1,0)=0$. In the latter case, by \cite[1.42]{mmp} again, $E$ is given by exceptional divisor of the blowup of non-reduced components of $W$ with smaller multiplicity. Hence if $X$ is a partial terminal modification of $Y$, then $X$ is isomorphic to either a successive blowup $X_1$ of $\mathbb{P}^{n}$ along hypersurfaces in the strict transform $\bar{H}$ of $H$, or the contraction of $\bar{H}$ from $X_1$.

\item Similarly, if $Y$ is the blow-up of $\bP^n$ along a hypersurface of degree $d\le n$ in a hyperplane $H$, then its partial terminal modification $X$ is isomorphic to a successive blowup of $\mathbb{P}^{n}$ along hypersurfaces in the strict transform of $H$;

\item Let $Y$ be a quartic weighted hypersurface in $\mathbb{P}(1^n,2^2)$ or the quotient of a quadric by an involution, then it is straightforward to check using the explicit equations in \cite{lz} that $Y$ has canonical singularity. Thus a partial terminal modification is given by extracting some divisors with discrepancy zero. For example, if $Y$ is the hypersurface $(x_n x_{n+1}=f^4)\subseteq\mathbb{P}(1^n,2^2)$ where $f$ is linear in $x_0,\cdots,x_{n-1}$ then there are three such exceptional divisors.

\item Similar if $Y$ is a Gorenstein del Pezzo surface of degree $\ge4$, then a partial terminal modification $X$ is given by Gorenstein weak del Pezzo surface (i.e. $-K_X$ is nef and big) of degree $\ge4$.
\end{enumerate}

\section{Izumi's inequality} \label{sec:izumi}

In this section we introduce some ingredient in the proof of Theorem \ref{main:blowupPn} that might be of independent interest. We start by giving another proof of Theorem \ref{main:P^n} using the techniques developed in \S \ref{sec:birbdd}.

\begin{proof}[Second proof of Theorem \ref{main:P^n}]
Let $Y$ be a terminal modification of $X$. By Lemma \ref{lem:mmodel}, we also have $\epsilon_m(-K_Y)>n$. Clearly it suffices to show that $Y\cong \bP^n$. Replacing $X$ by $Y$, we may assume that $X$ is $\bQ$-factorial and has only terminal singularities.

We proceed by induction on $n$. The result is clear when $n=1$. 

Since $-K_{X}$ is big, by Theorem \ref{thm:mmp} we may run the $K_{X}$-MMP $\pi:X\dashrightarrow Y$ which terminates with a Mori fiber space $g:Y\rightarrow Z$.
Let $F$ be the general fiber of $g$. By Corollary \ref{cor:nondecrease}
we have $\epsilon_{m}(-K_{Y})\ge\epsilon_{m}(-K_{X})>n$ hence $\epsilon_{m}(-K_{F})>n$ by
Lemma \ref{lem:restrict_m}. If $Z$ is not a point then by induction hypothesis we have $\epsilon_{m}(-K_{F})\le\epsilon_{m}(-K_{\bP^{n-1}})=n$, a contradiction. Thus $Z$ is a point and $Y$ is $\mathbb{Q}$-Fano. Since $-K_{Y}$ is ample in this case, we have $\epsilon_{m}(-K_{Y})=\epsilon(-K_{Y})$, hence $Y\cong\mathbb{P}^{n}$ by Theorem \ref{thm:P^n-weak}.

It remains to show that if $\pi:X\rightarrow\mathbb{P}^{n}$ is a
divisorial contraction then $\epsilon_{m}(-K_{X})\le n$. As $\bP^n$ does not admit flip, this will imply that the MMP $\pi:X\dashrightarrow Y\cong\bP^n$ is trivial, hence $X\cong \bP^n$ as well.

Since $\mathbb{P}^{n}$ is smooth and in particular has terminal singularities, we have $K_{X}\sim\pi^{*}K_{\mathbb{P}^{n}}+E$
for some $\pi$-exceptional effective divisor $E\neq0$. Let $\mathcal{I}_{k}=\pi_{*}\mathcal{O}_{X}(-kE)$ and let $x\in Z=\mathrm{supp}\,\pi(E)$.

\begin{claim*}
$\mathcal{I}_{k}\subseteq\mathfrak{m}_{x}^{k}$ for all $k\ge0$.
\end{claim*}

\begin{proof}[Proof of claim]
If this fails, then locally at $x$ there exists $f\in\mathcal{I}_{k}$
such that $\mathrm{mult}_{x}(f)<k$. Let $D_{0}=V(f)$, then we have
$\pi^{*}D_{0}\sim_{\mathbb{Q}}\bar{D}_{0}+dE$ where $\bar{D}_{0}$
is the strict transform of $D_{0}$ and $d\ge k$. Let $D=\frac{1}{d}D_{0}$ and $\bar{D}=\frac{1}{d}\bar{D}_{0}$,
then $\mathrm{mult}_{x}(D)<\frac{k}{d}\le1$ and $K_{X}+\bar{D}\sim_{\mathbb{Q}}\pi^{*}(K_{\mathbb{P}^{n}}+D)$.
In particular, the pair $(\mathbb{P}^{n},D)$ is not terminal. But
this contradicts Lemma \ref{lem:terminal}.
\end{proof}

Returning to the proof the theorem. By definition $s(-mK_{X})=s(\omega_{\mathbb{P}^{n}}^{-m}\otimes\mathcal{I}_{m})$
 and by the above claim $s(\omega_{\mathbb{P}^{n}}^{-m}\otimes\mathcal{I}_{m})\le s(\omega_{\mathbb{P}^{n}}^{-m}\otimes\mathfrak{m}_{x}^{m})$.
Let $\pi_{1}:X_{1}\rightarrow\mathbb{P}^{n}$ be the blowup of $x$,
$E_{1}$ the exceptional divisor and $L=\pi_{1}^{*}(-K_{\mathbb{P}^{n}})-E_{1}$,
then we also have $s(mL)=s(\omega_{\mathbb{P}^{n}}^{-m}\otimes\mathfrak{m}_{x}^{m})$.
Hence $s(mL)\ge s(-mK_{X})$ and $\epsilon_{m}(L)\ge\epsilon_{m}(-K_{X})$.
Note that $L$ is ample. It is straightforward to compute that $\epsilon(L)=\epsilon_m(L)=n$. This completes the proof of the theorem.
\end{proof}

The following result should be well known to experts. We include a proof here for reader's convenience.

\begin{lem} \label{lem:terminal}
Let $X$ be a smooth variety and $D$ an effective $\mathbb{Q}$-divisor on $X$. Assume $\mathrm{mult}_{x}D\le 1$ for some $x\in X$.
\begin{enumerate}
    \item If $\mathrm{mult}_{x}D<1$, then the pair $(X,D)$ is terminal in a neighbourhood of $x$;
    \item If $Z$ is a center of canonical singularity of $(X,D)$ containing $x$, then $Z\subseteq D$ and $Z$ has codimension 2 in $X$.
\end{enumerate}
\end{lem}

\begin{proof}
We prove by induction on $n$, the dimension of $X$. In the surface case, this follows from \cite[Theorem 4.5]{km98}, so we may assume $n\ge 3$.

Suppose $(X,D)$ is not terminal, hence has non-positive discrepancy along an exceptional divisor $E$ whose center $Z$ on $X$ contains $x$. If $\dim Z\ge 1$, let $H\subseteq X$ be a general hyperplane section (not necessarily containing $x$) and let $\pi:Y\rightarrow X$ be a log resolution of $(X,D)$ such that $E$ is $\pi$-exceptional. Since $H$ is general, this is also a log resolution of $(X,D+H)$ and $H'=\pi^*H$ is smooth. We may write $K_Y+H'=\pi^*(K_X+D+H)+a(E;X,D)E+\Delta$ where $E\not\subseteq\Supp\Delta$. By adjunction we get
\[K_{H'}=\pi^*(K_H+D|_H)+a(E;X,D)(E\cdot H')+(\Delta\cdot H).\]
As $\dim Z\ge 1$ and $H$ is general, $(E\cdot H')=(E\cdot \pi^*H)\neq 0$ is smooth, $\pi|_{H'}$-exceptional and is different from any component of $(\Delta\cdot H)$. By induction hypothesis, $(H,D|_H)$ is canonical and is terminal if $\mathrm{mult}_{x}D<1$. It follows that $a(E;X,D)\ge 0$ and it is strictly positive if $\mathrm{mult}_{x}D<1$. Moreover, if $a(E;X,D)=0$ then $Z\cap H=\pi(E\cap H')$ is a center of canonical singularity of $(H,D|_H)$. By induction hypothesis, $Z\cap H$ has codimension two in $H$, hence $\mathrm{codim}_X Z=2$ as well. It is clear that $Z\subseteq D$.

If $\dim Z=0$ then $Z$ is the point $x$. By \cite[Lemma 2.45]{km98}, $E$ is obtained by successive blowups of its center. In other words, there exists a sequence $X_0=X$, $Z_{i}=\mathrm{Center}_{X_i}(E)$ and $X_{i+1}=\mathrm{Bl}_{Z_i}X_i$ for $i=0,1,\cdots,m$ such that $Z_m$ is a divisor. Let $D_i$ be the strict transform of $D$ on $X_i$ and $E_i$ be the excpetional divisor of $f_i:X_i\rightarrow X_{i-1}$. As $Z_0={x}$, it is not hard to see that $\mathrm{mult}_{y}D_1\le \mathrm{mult}_{x}D\le 1$ for all $y\in E_1$ and we have
\[K_{X_1}+D_1=f_1^*(K_X+D)+aE_1\]
where $a=n-1-\mathrm{mult}_{x}D\ge 2-\mathrm{mult}_{x}D>0$. By induction on $m$, we may assume $(X_1,D_1)$ is canonical, hence as $a>0$, we have $a(E;X,D)>a(E;X_1,D_1)\ge 0$, thus $Z$ is not a center of canonical singularity. This proves the lemma.
\end{proof}

The claim that apppears in the above proof of Theorem \ref{main:P^n} can be rephrased as follows:

\begin{lem} \label{lem:inequality}
Let $x$ be a smooth point of $X$ and $\nu$ a divisorial valuation over $X$ whose center has codimension at least $2$ in $X$ and contains $x$, then we have
\[\nu(\mathfrak{m}_x)\mult_x\le\nu\le a(\nu)\mult_x\]
where $a(\nu)$ is the discrepancy of $\nu$.
\end{lem}

This is actually a refinement of the well-known Izumi-type inequality over a smooth point (see e.g. \cite[Theorem 2.6]{izumi-1} and \cite[Proposition 5.10]{izumi-2}), where the previous known bound on the right hand side of the inequality is $A(\nu)\mult_x$ (here $A(\nu)=1+a(\nu)$ is the log discrepancy of $\nu$). Note that the inequality is now optimal since we have an equality when $X$ is a surface and $\nu=\mult_x$. 

While such an inequality is sufficient for characterization of $\bP^n$, it is not enough for studying varieties $X$ coming from a blowup of $\bP^n$ with $\epsilon_m(-K_X)=n$. From the proof of Theorem \ref{main:P^n}, we see that we need to analyse the case when the constant in Izumi's inequality is optimal. This will be the goal of this section.

We first introduce some notations. Let $X$ be a smooth variety (not necessarily projective throughout this section). By \cite[Lemma 2.45]{km98}, every exceptional
divisor $E$ over $X$ with center $Z$ can be obtained by successive
blowups of its center, i.e. there exists a finite sequence $X_{0}=X$,
$Z_{0}=Z$, $X_{i}=\mathrm{Bl}_{Z_{i-1}}X_{i-1}$, $Z_{i}=\mathrm{Center}_{X_{i}}(E)$
($i=1,\cdots,m$) such that $Z_{m}$ is a divisor on $X_{m}$. We
call $m$ the length of $E$ over $X$. As before, let $\mathcal{I}_{k,E}=\{f\in\mathcal{O}_{X}|\nu_{E}(f)\ge ak\}\subseteq\mathcal{O}_{X}$
where $\nu_{E}$ is the divisorial valuation corresponding to $E$
and $a=a(E;X,0)$. We will simply write $\mathcal{I}_k$ if the choice of $E$ is clear.

\begin{lem} \label{lem:subpair}
Let $(X,\Delta=A-B)$ be a subpair where $X$
is smooth, $A,B$ are effective and have no common components. Let
$E$ be an exceptional divisor of length $m$ over $X$ and $x$ a general point of its center $Z$. Suppose that $\mathrm{mult}_{x}A\le1+a$
and $\mathrm{mult}_{x}B\ge b\ge(m-1)a$, then $a(E;X,\Delta)\ge b-ma$.
\end{lem}

\begin{proof}
We prove by induction on both $\dim X$ and $m$. Suppose $\dim Z\ge1$,
let $H\subseteq X$ be a general hyperplane section containing $x$
and $E',A',B',\Delta'$ the corresponding restriction to $H$. The
assumption on multiplicity doesn't change as $x$ is a general point
in $Z$. By induction hypothesis we then have $a(E;X,\Delta)=a(E';H,\Delta')\ge b-ma$.

Hence we may assume $Z$ is the point $x$. Let $\phi:X_{1}\rightarrow X$
be the blowup of $x$ with exceptional divisor $E_{1}$ and let $A_{1}$,
$B_{1}$ be the strict transform of $A$,\textbf{ $B$}. We have
\[
K_{X_{1}}+A_{1}-B_{1}\sim_{\mathbb{Q}}\phi^{*}(K_{X}+A-B)+a_{1}E_{1}
\]
where $a_{1}=n-1-\mathrm{mult}_{x}A+\mathrm{mult}_{x}B\ge b-a$. If
$m=1$ then $E=E_{1}$ and $a(E;X,\Delta)=a_{1}\ge b-a$ as required.
Suppose $m>1$, then $a_{1}\ge0$ by assumption, and we have $a(E;X,\Delta)=a(E;X_{1},A_{1}-B_{1}-a_{1}E_{1})\ge a(E;X_{1},A_{1}-a_{1}E_{1})$.
Since for any $y\in E_{1}$ we also have $\mathrm{mult}_{y}A_{1}\le\mathrm{mult}_{x}A\le1+a$
and $\mathrm{mult}_{x}(a_{1}E_{1})=a_{1}\ge b-a$, by induction hypothesis
we get $a(E;X_{1},A_{1}-a_{1}E_{1})\ge b-a-(m-1)a\ge b-ma$. This
completes the proof.
\end{proof}

\begin{lem} \label{lem:izumi}
Let $X$ be a smooth variety, $E$ a divisor over $X$ with center $Z$ and length $m$ and $x$ a general point in $Z$. Suppose $Z$ has codimension at least 3. Then $\mathcal{I}_k=\mathcal{I}_{k,E}\subseteq\mathfrak{m}_{x}^{\lambda k}$ where $\lambda=1+\frac{1}{m}$.
\end{lem}

\begin{proof}
By taking general hyperplane sections it suffices to consider the
case when $Z$ is a point and $X$ is affine. Suppose $\mathcal{I}_{k}\not\subseteq\mathfrak{m}_{x}^{\lambda k}$,
then there exists $f\in\mathcal{O}_{X}$ with $\nu_{E}(f)\ge ak$
and $\mathrm{mult}_{x}f<\lambda k$. Let $D_{0}$ be the zero locus
of $f$ and $D=\frac{1}{k}D_{0}$, we then have $\mathrm{mult}_{x}D<\lambda=1+\frac{1}{m}$
and $a(E;X,D)\le0$. Let $\phi:X_{1}\rightarrow X$ be the blow up
of $x$, $E_{1}$ the exceptional divisor and $D_{1}$ the strict
transform of $D$, we have $K_{X_{1}}+D_{1}-a_{1}E_{1}\sim_{\mathbb{Q}}\phi^{*}(K_{X}+D)$
where $a_{1}=n-1-\mathrm{mult}_{x}D>1-\frac{1}{m}$ (since $n\ge3$
by assumption) and $\mathrm{mult}_{y}D_{1}\le\mathrm{mult}_{x}D<1+\frac{1}{m}$ for all $y\in E_1$.
Since $E$ has length $m-1$ over $X_{1}$, we may apply Lemma \ref{lem:subpair}
to the subpair $(X_{1},\Delta_{1}=D_{1}-a_{1}E_{1})$ to get $a(E;X,D)=a(E;X_{1},\Delta_{1})>1-\frac{1}{m}-(m-1)\cdot\frac{1}{m}=0$,
a contradiction. Hence $\mathcal{I}_{k}\subseteq\mathfrak{m}_{x}^{\lambda k}$
holds.
\end{proof}

\begin{lem} \label{lem:surface}
Let $X=\mathbb{A}_{K}^{2}$ where $K$ is a $($not
necessarily algebraically closed$)$ field of characteristic zero and
$\nu$ a divisorial valuation centered over $x=(0,0)$. Define $\mathcal{I}_{k}\subseteq\mathcal{O}_{X}$
as before. Suppose for all $\lambda>1$, we have $\mathcal{I}_{k}\not\subseteq\mathfrak{m}_{x}^{\lambda k}(k\gg0)$,
then $\nu$ is a monomial valuation given by $\nu(s)=1$, $\nu(t)=m$
under certain coordinate $s,t$ of $X$.
\end{lem}

\begin{proof}
First assume $K=\bar{K}$. Choose coordinate $s,t$ such that $\nu(t)\ge\nu(s)=N>0$. 
By \cite[Proposition 4.1]{fj04}, there exists
a finite generic Puiseux series, i.e. some $\phi(s)=\sum_{i=1}^{r}a_{i}s^{\beta_{i}}+\xi\cdot s^{\beta}$ (where $\xi$ is an indeterminate that can be viewed as the local
coordinate on $E$ and $1\le\beta_{1}<\cdots<\beta_{r}<\beta$ satisfy
$N\beta_{i}\in\mathbb{Z}$ and $N\beta\in\mathbb{Z}$) such that $\nu(f)=N\cdot\mathrm{mult}_{s}f(s,\phi(s))$ (it is clear that for any finite generic Puiseux series given as above, $\nu(f)=N\cdot\mathrm{mult}_{s}f(s,\phi(s))$ defines a discrete valuation on $X$).

\begin{claim}
$\nu(f)\le N\beta\mathrm{mult}_{x}f$ for all $f\in K[s,t]$.
\end{claim}

\begin{proof}[Proof of claim]
Let $\hat{K}$ be th algebraic closure of $\mathbb{C}((s))$. Its
elements are finite or infinite Puiseux series of the form 
\[
\hat{\phi}(s)=\sum_{j\ge1}a_{j}s^{\hat{\beta}_{j}}\quad\mathrm{with}\quad a_{j}\in\mathbb{C}^{*},\hat{\beta}_{j+1}>\hat{\beta}_{j}\in\mathbb{Q}
\]
where the rational numbers $\hat{\beta}_{j}$ have bounded denominators.
By \cite[Proposition 4.1]{fj04} again the valuation $\nu$ extends to
a valuation on $\hat{K}[[t]]$ and it suffices to prove the claim
for those $f\in\hat{K}[[t]]$ that is linear on $t$. If $f\in\hat{K}$
then $\nu(f)=N\cdot\mathrm{mult}_{x}f$ and the claim follows since
$\beta\ge1$. If $f=t-\psi$ where $\psi\in\hat{K}$, then either
$\mathrm{mult}_{s}\psi<1$ and $\nu(f)=N\cdot\mathrm{mult}_{s}\psi\le N\beta\mathrm{mult}_{x}f$
as before (here we use the assumption that $\mathrm{mult}_{x}\phi\ge1$)
or $\mathrm{mult}_{x}f=1$ and $\nu(f)=N\cdot\mathrm{mult}_{s}(\phi-\psi)\le N\beta$.
Again the claim follows.
\end{proof}
Returning to the proof of the lemma. To compare $\mathcal{I}_{k}$
and $\mathfrak{m}_{x}^{\lambda k}$, we also need to compute the discrepancy
$a=a(E;X,0)$. For this we have
\[a=\nu(\mathrm{d}s\wedge\mathrm{d}t)=\nu(\mathrm{d}s\wedge\mathrm{d}\phi(s))=\nu(s^{\beta}\mathrm{d}s\wedge\mathrm{d}\xi)\ge\beta\nu(s)+\nu(\mathrm{d}s)\ge N\beta+N-1.\]
Combining this with the above claim we have $\mathcal{I}_{k}\subseteq\mathfrak{m}_{x}^{\mu k}$
where $\mu=1+\beta^{-1}(1-\frac{1}{N})$. By assumption $\mu\le1$,
hence $N=1$ and $\phi$ is a polynomial in $s$. After the change
of coordinate $s'=s$, $t'=t-\sum_{i=1}^{r}a_{i}s^{\beta_{i}}$, $\nu$
becomes the monomial valuation given by $\nu(s')=1$, $\nu(t')=\beta$. 

Now we treat the general case where $K$ may not be algebraically
closed. Let $K[X]=K[s,t]$ where $0<\nu(s)\le\nu(t)$. $\nu$ extends
to a divisorial valuation over $\bar{K}$ and let $\bar{\mathcal{I}}_{k}\subseteq\mathcal{\mathcal{O}}_{X_{\bar{K}}}$
be the corresponding ideal sheaf defined in the same way as $\mathcal{I}_{k}$.
By what we have shown above, there exists $f(s)\in\bar{K}[s]$ and
$m\in\mathbb{N}$ such that $f(0)=0$, $\deg f<m$ and $\nu$ is the
monomial valuation $\nu(s)=1$, $\nu(y)=m$ where $y=t-f(s)$. A direct
computation yields $\bar{\mathcal{I}}_{k}=(s^{m},y)^{k}$. If $m=1$,
then $f=0$, $y=t$ and $\nu$ is already a monomial valuation under
the original coordinate $(s,t)$. If $f(s)\in K[s]$, then $(s,y)$
is a new coordinate under which $\nu$ becomes monomial. So we may
assume $m\ge2$ and $f\not\in K[s]$.

\begin{claim}
Under these assumptions, there exists $\lambda>1$ such that $\mathcal{I}_{k}=\bar{\mathcal{I}}_{k}\cap K[s,t]\subseteq\mathfrak{m}_{x}^{\lambda k}$
for $k\gg0$.
\end{claim}

\begin{proof}[Proof of claim]
After a further change of coordinate (subtract the terms of $f(s)$
with coefficients in $K$ from $t$) we may assume $y=t-as^{l}+\mathrm{H.O.T.}$
where $l<m$ and $a\in\bar{K}\backslash K$. We then have $\bar{\mathcal{I}}_{k}=(s^{m},y)^{k}\subseteq(s^{l+1},t-as^{l})^{k}$,
hence it suffices to prove the claim when $f(s)=as^{m-1}$.

Let $s,t$ have weights $1,m-1$ respectively and let $\bar{K}[s,t]$
be graded according to weighted degrees of polynomials. In particular,
the generators $s^{m}$, $y$ of $\bar{\mathcal{I}}_{1}$ have weighted
degree $\mathrm{wt}(s^{m})=m$, $\mathrm{wt}(y)=m-1$, respectively.
Let $g(s,t)\in\bar{\mathcal{I}}_{k}\cap K[s,t]$. We may express $g$
as a sum $g=\sum g_{i}$ where $g_{i}$ is weighted homogeneous of
degree $i$. Let $g_{\le l}=\sum_{i\le l}g_{i}$. Since $\bar{\mathcal{I}}_{k}$
is generated by weighted homogeneous elements, we have $g_{\le l}\in\bar{\mathcal{I}}_{k}\cap K[s,t]$
for all $l\in\mathbb{Z}$ and $g_{\le l}$ is indeed contained in
the ideal generated by those $(t-as^{m-1})^{p}s^{m(k-p)}$ with weighted
degree $\le l$, i.e. $p\ge mk-l$. Suppose $l<mk$. It then follows
that $(t-as^{m-1})^{mk-l}$ divides $g_{\le l}$. Since $a\not\in K$
and $g_{\le l}\in K[s,t]$, the conjugates of $(t-as^{m-1})^{mk-l}$
also divides $g_{\le l}$. Hence if $g_{\le l}\ne0$ we have $l\ge\mathrm{wt}(g_{\le l})\ge2(m-1)(mk-l)$
and $l\ge\frac{2m(m-1)}{2m-1}k$. In other words we have $\nu_{0}(g)\ge\frac{2m(m-1)}{2m-1}k$
where $\nu_{0}$ is the monomial valuation given by $\nu_{0}(s)=1$,
$\nu_{0}(t)=m-1$. By Lemma \ref{lem:inequality} we have $\nu_{0}(g)\le(m-1)\mathrm{mult}_{x}(g)$ (the discrepancy of $\nu_{0}$ is $m-1$), hence $\mathrm{mult}_{x}(g)\ge\frac{2m}{2m-1}k$
and the claim follows by taking $\lambda=\frac{2m}{2m-1}>1$.
\end{proof}

Hence if $m\ge2$ and $f\not\in K[s]$ then the valuation $\nu$ does
not satisfy the assumption of the lemma. This concludes the proof.
\end{proof}

\begin{lem} \label{lem:equality}
Let $X$ be a smooth variety, $E$ an exceptinoal divisor over $X$ with center $Z$ and $x$ a general point in $Z$. Suppose for all $\lambda>1$,
we have $\mathcal{I}_{k}\not\subseteq\mathfrak{m}_{x}^{\lambda k}(k\gg0)$,
then $Z$ has codimension two in $X$ and $E$ is obtained as the
last exceptional divisor of a successive blowup $X_{m}\rightarrow\cdots\rightarrow X_{0}=X$
where the center $Z_{i}$ of each blowup $X_{i+1}\rightarrow X_{i}$
maps birationally to $Z$ and is not contained in the excpetional
divisors of $X_{i-1}\rightarrow X_{0}$.
\end{lem}

\begin{proof}
$Z$ has codimension two by Lemma \ref{lem:izumi}. The second statement
is local (in the analytic topology of $X$), so after localizing at the generic point $x$ of $Z$, we
may assume $X=\mathbb{A}_{K}^{2}$ where $K$ is the residue field
of $\mathcal{O}_{X,x}$ and $Z=\{x\}$. By Lemma \ref{lem:surface},
$\nu=\nu_{E}$ is a monomial valuation given by $\nu(s)=1$, $\nu(t)=m$
under certain coordinate $s,t$ of $X$. In other words, $E$ is obtained
by successively blowing up the intersection point of the strict transform
of the smooth curve $(t=0)$ and the last exceptional divisor. Translating
this to the origin variety $X$ gives the statement of the lemma.
\end{proof}

\section{Moving Seshadri constant on blowup of projective space} \label{sec:blowup}

We now take up the study of varieties $X$ with a birational contraction to $\bP^n$ such that $\epsilon_{m}(-K_{X})=n$. In particular we prove Theorem \ref{main:blowupPn}. We assume $n\ge3$ throughout this section.



We first establish an auxiliary lemma. Recall that the secant variety $S(Z)$ of $Z\subseteq\mathbb{P}^{n}$ is just the closure of the union of all secant lines (i.e. lines in $\mathbb{P}^{n}$ that intersect $Z$ at at least two points).

\begin{lem} \label{lem:secant}
Let $n\ge3$ and $Z\subseteq\mathbb{P}^{n}$ a
reduced subscheme of pure codimension two. Assume that $S(Z)\neq\mathbb{P}^{n}$,
then one of the following holds:
\begin{enumerate}
\item $Z$ is contained in a hyperplane;
\item There exists a linear subspace $V$ of codimension 3 such that every
irreducible component of $Z$ is a linear subspace containing $V$.
\end{enumerate}
\end{lem}

\begin{proof}
For $p\in\mathbb{P}^{n}$ let $C_{p}(Z)$ be the cone over $Z$ with
vertex $p$ (if $p\in Z$ this is the closure of the union of
secant lines containing $p$). First assume that $Z$ is irreducible
and not a linear subspace. Then for general points $p,q\in Z$ the
line $\overline{pq}$ joining $p$ and $q$ is not contained in $Z$
and since $S(Z)\neq\mathbb{P}^{n}$ we have $\dim S(Z)=n-1$. Since
$S(Z)$ is irreducible in this case, we also have $S(Z)=C_{p}(Z)$
for a general point $p\in Z$. It follows that if $p\in Z$ and $q\in S(Z)$
($p\neq q$) then $\overline{pq}\subseteq S(Z)$. Hence we also have
$S(Z)=C_{q}(Z)$ for a general point $q\in S(Z)$. But then we have
$\overline{pq}\subseteq S(Z)$ for any two points $p\neq q\in S(Z)$,
and $S(Z)$ has to be a linear subspace, hence a hyperplane for dimension
reason. By a similar argument, if $Z_{1}$ and $Z_{2}$ are two irreducible
component of $Z$, then their join (i.e. the closure of the union
of lines $\overline{pq}$ where $p\in Z_{1}$, $q\in Z_{2}$) is also
a hyperplane.

Now let $Z_{1},\cdots,Z_{r}$ be the irreducible components of $Z$.
If one of the components, say, $Z_{1}$ is not a linear subspace,
then it is contained in a unique hyperplane $H=S(Z_{1})$. Since $Z_{i}$
and $Z_{1}$ are also contained in a hyperplane for all $i$, we must
have $Z_{i}\subseteq H$ as well, thus $Z$ is contained in the hyperplane
$H$. If all the components $Z_{i}$ are linear subspaces, then for
every triple $(Z_{1},Z_{2},Z_{3})$ either they are contained in a
hyperplne or we have 
\[
Z_{i}\cap Z_{j}=H_{1}\cap H_{2}\cap H_{3}=Z_{1}\cap Z_{2}\cap Z_{3}\,(\forall i,j)
\]
where $H_{i}$ is the hyperplane containing the subspaces except $Z_{i}$.
We have $V=\cap_{i=1}^{3}H_{i}\neq\emptyset$ since $n\ge3$. It is
then not hard to see that either all $Z_{i}$'s are contained in a
hyperplane or they all contain a common linear subspace $V$ of codimension
three.
\end{proof}

Now we give the proof of Theorem \ref{main:blowupPn}.

\begin{proof}[Proof of Theorem \ref{main:blowupPn}]
Let $E$ be an $f$-exceptional divisor and $Z$ its center on $Y=\mathbb{P}^{n}$.
As in previous sections, let $\mathcal{I}_{k}=\mathcal{I}_{k,E}=\{f\in\mathcal{O}_{Y}|\nu_{E}(f)\ge ak\}\subseteq\mathcal{O}_{Y}$
where $\nu_{E}$ is the divisorial valuation corresponding to $E$
and $a=a(E;Y,0)$. The birational contraction $f$ induces an inclusion
$H^{0}(X,\mathcal{O}_{X}(-kK_{X}))\rightarrow H^{0}(Y,\mathcal{O}_{Y}(-kK_{Y}))$
whose image is contained in $H^{0}(Y,\mathcal{I}_{k}(-kK_{Y}))$,
hence we have
\begin{equation}
\limsup_{k\rightarrow\infty}\frac{s(\mathcal{I}_{k}(-kK_{Y}))}{k}\ge\epsilon_{m}(-K_{X})=n\label{eq:2}
\end{equation}

Suppose there exists $\lambda>1$ and $x\in Z$ such that $\mathcal{I}_{k}\subseteq\mathfrak{m}_{x}^{\lambda k}$
for $k\gg0$, then by the same argument as in the proof of Theorem \ref{main:P^n}
we have
\[
\limsup_{k\rightarrow\infty}\frac{s(\mathcal{I}_{k}(-kK_{Y}))}{k}\le\limsup_{k\rightarrow\infty}\frac{s(\mathfrak{m}_{x}^{\lambda k}(-kK_{Y}))}{k}\le n+1-\lambda<n
\]
which contradicts (\ref{eq:2}). It follows that for any $\lambda>1$
we have $\mathcal{I}_{k}\not\subseteq\mathfrak{m}_{x}^{\lambda k}$
for $k\gg0$, thus by Lemma \ref{lem:equality}, $Z$ has codimension
two. Let $W$ be the union of all centers of $f$-exceptional divisors,
then $W$ is a subscheme of pure codimension two and by Lemma \ref{lem:inequality} we have $\mathcal{I}_{k}\subseteq\mathcal{I}_{W}^{(k)}$ (the $k$-th symbolic power of $\cI_W$).
Analogous to (\ref{eq:2}) we have
\begin{equation}
\limsup_{k\rightarrow\infty}\frac{s(\mathcal{I}_{W}^{(k)}(-kK_{Y}))}{k}\ge\epsilon_{m}(-K_{X})=n\label{eq:3}
\end{equation}

\begin{claim}
$W$ is contained in a hyperplane.
\end{claim}

\begin{proof}
First suppose that $S(W)=\mathbb{P}^{n}$, then for a very general
point $p\in\mathbb{P}^{n}$, there exists $x\neq y\in W$ such that
$p\in\overline{xy}$. Let $\sigma:\hat{Y}\rightarrow Y$ be the blowup
of $x$ and $y$ with exceptional divisors $E_{x}$, $E_{y}$. Since
$\mathcal{I}_{W}\subseteq\mathfrak{m}_{x}\otimes\mathfrak{m}_{y}$,
by (\ref{eq:3}) we have
\[
\epsilon_{m}(L,p)=\limsup_{k\rightarrow\infty}\frac{s(\mathfrak{m}_{x}^{k}\otimes\mathfrak{m}_{y}^{k}(-kK_{Y}))}{k}\ge n
\]
where $L=\sigma^{*}(-K_{Y})-E_{x}-E_{y}$. On the other hand, as $L$
is nef, we have $\epsilon_{m}(L,p)=\epsilon(L,p)\le(L\cdot l)=n-1$
where $l$ is the strict transform of the line $\overline{xy}$, a
contradiction. Hence $S(W)\neq\mathbb{P}^{n}$.

Now by Lemma \ref{lem:secant}, either $W$ is contained in a hyperplane
or there exists a linear subspace $V$ of codimension 3 such that
$W$ is a union of linear subspaces containing $V$. Suppose $W$
is not contained in a hyperplane (so we are in the second case), then
it is easy to see that $\mathcal{I}_{W}\subseteq\mathcal{I}_{V}^{2}\subseteq\mathfrak{m}_{x}^{2}$
for any $x\in V$, hence argue as before we have
\[
\limsup_{k\rightarrow\infty}\frac{s(\mathcal{I}_{W}^{k}(-kK_{Y}))}{k}\le\limsup_{k\rightarrow\infty}\frac{s(\mathfrak{m}_{x}^{2k}(-kK_{Y}))}{k}\le n-1
\]
which contradicts (\ref{eq:3}). The claim then follows.
\end{proof}

In particular, $Z$ is also contained in a hyperplane, hence is a
complete intersection in $\mathbb{P}^{n}$.

\begin{claim}
There exists a hyperplane $H$ containing $Z$ such that $E$ is obtained
by the blowup of a hypersurface in $H$.
\end{claim}

\begin{proof}
By Lemma \ref{lem:equality}, $E$ is obtained as the last exceptional divisor of a successive
blowup $X_{m}\rightarrow\cdots\rightarrow X_{0}=Y$ such that for
$i=0,\cdots,m-1$, the center $Z_{i}$ of each blowup $X_{i+1}\rightarrow X_{i}$
maps birationally to $Z$ and is not contained in $E_{1},\cdots,E_{i-1}$
where $E_{j}$ is the exceptional divisor of $X_{j}\rightarrow X_{j-1}$.
In particular $E=E_{m}$. After localizing at the generic point of
$Z$ it is not hard to see that there is a crepant birational contraction
$g_{i}:X_{i}\dashrightarrow\bar{X}_{i}$ for each $i$ that contracts
$E_{1},\cdots,E_{i-1}$. Since $\bar{X}_{m}\dashrightarrow Y=\mathbb{P}^{n}$
extracts exactly the divisor $E$, we have $\epsilon_{m}(-K_{\bar{X}})\ge n$
by (\ref{eq:3}). Since $g_{m}$ is crepant, i.e. $g_{m}^{*}K_{\bar{X}_{m}}=K_{X_{m}}$,
we have $\epsilon_{m}(-K_{X_{m}})\ge n$ as well. By Corollary \ref{cor:nondecrease},
$\epsilon_{m}(-K_{X_{i}})\ge n$ for $i=1,\cdots,m$. 

We now prove the claim by induction on $m$. If $m=1$, then $X_{1}$
is just the blowup of $Z$, which is a hypersurface in a hyperplane.
Hence in what follows, assume $m>1$.

By induction hypothesis, we may assume $h:\bar{X}_{m-1}\rightarrow\mathbb{P}^{n}$ 
is the blowup of a hypersurface $D$ contained in a hyperplane $H$ (a priori $h$ is only birational from the previous construction, but since moving Seshadri constants are preserved under small birational map, we may replace $h$ by the actual blowup of $D$).
By Lemma \ref{lem:equality}, $Z_{m-1}$ is a birational section of
the $\mathbb{P}^{1}$-bundle $E_{m-1}=\mathbb{P}(\mathcal{N}_{D/Y})$
over $D_{\mathrm{red}}$ such that $Z_{m-1}$ is contained in the
smooth locus of $\bar{X}_{m-1}$. Note that $\mathcal{N}_{D/Y}=\mathcal{O}_{D}(1)\oplus\mathcal{O}_{D}(d)$
where $d=\deg D$, since $D$ is a complete intersection. We also
have $-K_{\bar{X}_{m-1}}\sim h^{*}(nH)+H'$ where $H'$ is the
strict transform of $H$. Let $p$ be a very general point in $\mathbb{P}^{n}$,
$q\in D$ and $l_{pq}$ the strict transform of the line $\overline{pq}$
on $\bar{X}_{m-1}$. If $l_{pq}$ intersects $Z_{m-1}$ then by Lemma \ref{lem:sbound}
we have $\epsilon_{m}(-K_{\tilde{X}})\le(-K_{\tilde{X}}\cdot l_{pq}) <(-K_{\bar{X}_{m-1}}\cdot l_{pq})=n$
where $\tilde{X}$ is the blowup of $\bar{X}_{m-1}$ at $Z_{m-1}$;
on the other hand there is a crepant birational contraction $X_{m}\dashrightarrow\tilde{X}$,
hence $\epsilon_{m}(-K_{\tilde{X}})=\epsilon_{m}(-K_{X_{m}})\ge n$,
a contradiction. It follows that $Z_{m-1}$ is disjoint from the strict
transform $C_{p}(D)$ of the cone over $D$ with vertex $p$. Since
$C_{p}(D)$ intersects $E_{m-1}=\mathbb{P}(\mathcal{N}_{D/Y})$ at
the section corresponding to the surjection $\mathcal{N}_{D/Y}\twoheadrightarrow\mathcal{O}_{D}(d)$,
$Z_{m-1}$ must be a section corresponding to $\mathcal{N}_{D/Y}\twoheadrightarrow\mathcal{O}_{D}(1)$
(if $d=1$ there are infinitely many such surjections). In other words,
$Z_{m-1}$ is the intersection of $E_{m-1}$ with the strict transform
of a hyperplane (which is exactly $H$ if $d>1$). Hence $E=E_{m}$
is obtained by successively blowing up its center in the strict transform
of $H$. If $H=(h=0)$ and $Z=(h=f=0)$ then $\bar{X}_{m}$ is the
blowup of $(h=f^{m}=0)$ which is a still hypersurface in $H$. The
claim now follows.
\end{proof}

Let $E_{1},\cdots,E_{m}$ be the $f$-exceptional divisors with centers
$Z_{1},\cdots,Z_{m}$ in $\mathbb{P}^{n}$. By the previous claims,
we can choose hyperplanes $H$, $H_{1},\cdots,H_{m}$ such that $\cup Z_{i}\subseteq H$
and $E_{i}$ is obtained by blowing up a hypersurface $D_{i}$ in
$H_{i}$. We now show that it is possible to choose $H=H_{1}=\cdots=H_{m}$.

Notice that if $D_{i}$ is a linear subspace, then we may choose $H_{i}$
as any hyperplane that contains $D_{i}$. On the other hand if $d_{i}=\deg D_{i}\ge2$
then the hyperplane $H_{i}$ is unique. If for every $i=1,\cdots,m$
we have either $D_{i}$ is a linear subspace or $\deg Z_{i}\ge2$
then we can simply take $H_{i}=H$. Hence we may assume that $Z_{1}$
is a linear subspace, $d_{1}\ge2$ and proceed to show that we can
take $H_{i}=H_{1}$ for all $i>1$. Let $H_{1}=(h_{1}=0)$. There
are two cases to consider:

\subsection*{Case 1} $Z_{i}=Z_{1}$. If $d_{1}=1$ then we may just take $H_{i}=H_{1}$.
Assume $d_{i}\ge2$ and $H_{i}\neq H_{1}$. Let $H_{i}=(h=0)$, then
it is not hard to see that $\mathcal{I}_{D_{1}}\subseteq(h_{1},h^{2})$
and $\mathcal{I}_{D_{i}}\subseteq(h,h_{1}^{2})$, hence $\mathcal{I}_{D_{1}}^{k}\cap\mathcal{I}_{D_{i}}^{k}\subseteq(h,h_{1})^{\frac{4}{3}k}$.
Since $\mathcal{J}_{k}=\mathcal{I}_{k,E_1}\cap\mathcal{I}_{k,E_i}=\mathcal{I}_{D_{1}}^{k}\cap\mathcal{I}_{D_{i}}^{k}\subseteq\mathfrak{m}_{x}^{\frac{4}{3}k}$
for any $x\in Z_{1}=Z_{i}$, we have
\[
\epsilon_{m}(-K_{X})\le\limsup_{k\rightarrow\infty}\frac{s(\mathcal{J}_{k}(-kK_{Y}))}{k}\le\limsup_{k\rightarrow\infty}\frac{s(\mathfrak{m}_{x}^{\frac{4}{3}k}(-kK_{Y}))}{k}\le n+1-\frac{4}{3}<n
\]
a contradiction. Hence $H_{i}=H_{1}$.

\subsection*{Case 2} $Z_{i}\neq Z_{1}$. Since $n\ge3$ and each $Z_{i}$ has codimension
one in $H$, there exists some $y\in Z_{1}\cap Z_{i}$. Suppose $H\neq H_{1}$ and let $H=(h=0)$, then $Z_i=(h=f=0)$ for some $f$. We then also have $\cI_{Z_1}=(h_1,h)$, $\mathcal{I}_{D_{1}}\subseteq(h_{1},h^{2})$ and $\mathcal{I}_{D_{i}}\subseteq(h,f)$. Let $\mathcal{J}_{k}=\mathcal{I}_{D_{1}}^{k}\cap\mathcal{I}_{D_{i}}^{k}\subseteq (h,f)^k$
as in the first case and let $g\in\mathcal{O}_{Y,y}$ be an element
in $\mathcal{J}_{k}$. Since $g\in\mathcal{I}_{D_{1}}^{k}$ we may
write $g=\sum_{2p\ge k}a_{p}h_{1}^{k-p}h^{2p}+\sum_{2p<k}a_{p}h_{1}^{k-p}h^{2p}\coloneqq g_{1}+g_{2}$
where $a_{p}\in\mathcal{O}_{Y,y}$. It is clear that $g_{1}\in(h,f)^{k}$
(hence $g_{2}\in(h,f)^{k}$ as well) and $\mathrm{mult}_{y}(g_{1})\ge\frac{3}{2}k$.
On the other hand we have $g_{2}=h_{1}^{\frac{k}{2}}u$ for some $u\in\mathcal{O}_{Y,y}$.
Since $h,f,h_{1}$ form a regular sequence at $y$, by \cite[\S 15, Theorem 27]{matsumura} we have $u\in(h,f)^{k}$, hence $\mathrm{mult}_{y}(u)\ge k$ and $\mathrm{mult}_{y}(g_{2})\ge\frac{3}{2}k$.
It follows that $\mathrm{mult}_{y}(g)\ge\frac{3}{2}k$, hence $\mathcal{J}_{k}\subseteq\mathfrak{m}_{y}^{\frac{3}{2}k}$,
and a similar argument as in the first case yields $\epsilon_{m}(-K_{X})\le n-\frac{1}{2}<n$,
a contradiction. Therefore, $H=H_{1}$ and we have $H_i=H$ if either $D_i$ is a linear space or $\deg Z_i\ge 2$. If $Z_i$ is a linear space and $d_i=\deg D_i\ge 2$ then interchanging the role of $D_1$ and $D_i$ in the above proof we have $H=H_i$ as well. Thus in any case we have $H_1=H_i$.

\smallskip
Now that every $E_{i}$ is obtained as the blow up of a hypersurface in the same hyperplane $H$, $X$ is isomorphic in codimension one to a successive blowup of $\mathbb{P}^{n}$ along hypersurfaces in the strict transforms of $H$.
\end{proof}

\section{Examples and further questions} \label{sec:example}

In this last section we exhibit some examples and propose a few interesting further questions.

Our first question concerns the optimal bound $M(n,\epsilon)$ of $\vol(-K_X)$ in Theorem \ref{main:volbdd}. Recall that by the proof Theorem \ref{main:volbdd}, we can already take $M(n,\epsilon)=O(\frac{n^{2n}}{\epsilon^n})$. However, the calculation in Example \ref{exa:wp} suggests that an improvement might exist and the optimal bound should be given by $O(\frac{n^n}{\epsilon})$.

\begin{que}
Is is possible to take $M(n,\epsilon)=O(\frac{n^n}{\epsilon})$ in Theorem \ref{main:volbdd}?
\end{que}

Indeed, by the classification in \cite{lz} we have $M(n,1)=(n+1)^n$  when $\epsilon=1$ and the maximum is achieved by the projective space. On the other hand, in the surface case, as the anticanonical volume and moving Seshadri constant are both non-decreasing when taking minimal resolution and contracting $(-1)$-curves, it suffices to find $M(2,\epsilon)$ by examining all ruled surfaces and it follows that $M(2,\epsilon)=O(\epsilon^{-1})$. Interpolating these two results provides another evidence in favor of the above suggested optimal bound.

Along a different direction, in light of Theorem \ref{main:fanotype}, it is natural to ask:

\begin{que}
Let $X$ be a variety of dimension $n$ with $\epsilon_m(-K_X)\ge n$. Is it true that $X$ is of Fano type?
\end{que}

As mentioned in the introduction, a positive answer to this question would lead to a full classification of $X$ (up to small birational equivalence) with $\epsilon_m(-K_X)=n$.

Observe that by Therem \ref{main:birbdd} and the classification in \cite{lz}, all varieties $X$ with $\epsilon_m(-K_X)=n$ are rational. We may therefore expect rationality under a weaker assumption on Seshadri constant. In particular, we ask:

\begin{que}
Let $X$ be a variety of dimension $n$. Suppose $\epsilon_m(-K_X)>n-1$, is it true that $X$ is rational? By Theorem \ref{main:birbdd}, this is equivalent to asking: if $X$ is a terminal Fano variety whose blowup at a smooth point is still Fano, is it true that $X$ is rational?
\end{que}

Obviously the question is nontrivial only when $\dim X\ge 3$. By \cite[Theorem 1.1]{bcw}, if $X$ is a smooth Fano variety such that $\mathrm{Bl}_x X$ is still Fano then $X$ is rational. The following examples seem to provide further evidence for a positive answer to this question.

\begin{expl}
Let $X=X_{d}\subseteq Y=\mathbb{P}(1^{n},k,l)$ be a general weighted hypersurface of degree $d<n+k+l$ where $l\ge\max\{2,k\}$, let $r=d-k-l$ and let $m\le d$ be the largest integer of the form $ak+bl$ where $a,b\ge0$. Then $-K_{X}\sim(n-r)H$ where $H=(x_0=0)$ and we claim that
\begin{equation}
\epsilon(-K_{X})\le\frac{(n-r)m}{kl}\label{eq:4}
\end{equation}
with equality when $d\le kl$. To see this, let $x\in X$ be a very general point, $V\subseteq Y$
the linear subspace defined by $(x_{0}=\cdots=x_{n-1}=0)$ and $C_{x}(V)\subseteq Y$
the cone over $V$ with vertex $x$. Since $X$ is general, $C_{x}(V)\cap X$ contains an irreducible curve $C\sim\frac{m}{d} H^{n-1}$ that is smooth at $x$ (if $m<d$ then $V$ is the other component in the intersection). We thus have $\epsilon(-K_X,x)\le (-K_X\cdot C)=\frac{(n-r)m}{kl}$. If $m\le d\le kl$ then $\epsilon(-K_X)\le n-r$. Let $0\le t\le n-r$ and let $\sigma:\hat{X}\rightarrow X$
be the blowup of $x$ with exceptional divisor $E$, we have $L=\sigma^{*}(-K_{X})-tE\sim(n-r-t)\sigma^{*}H+t(\sigma^{*}H-E)$
is effective. Since $H$ is nef and $\mathrm{Bs}(\sigma^{*}H-E)$ is contained in the strict transform of $C\cup V$, $L$ is nef if and only if $(L\cdot C)=(-K_{X}\cdot C)-t\ge0$.
Hence the equality in (\ref{eq:4}) holds when $d\le kl$.

On the other hand, we have $(-K_{X})^{n}=\frac{(n-r)^{n}d}{kl}$.
If $r\ge2$ then $\left(\frac{n-1}{n-r}\right)^{n}=\left(1+\frac{r-1}{n-r}\right)^{n}>1+(r-1)+\frac{1}{2}(r-1)^{2}=\frac{1}{2}(r^{2}+1)\ge r+\frac{1}{2}$
and we obtain
\[
\frac{(-K_{X})^{n}}{(n-1)^{n}}<\frac{d}{(r+\frac{1}{2})kl}=\frac{r+k+l}{(r+\frac{1}{2})kl}\le\frac{r+k+l}{rkl+1}\le1
\]
hence $(-K_{X})^{n}<(n-1)^{n}$ and $\epsilon(-K_{X})<n-1$. If $r=1$
and $\epsilon(-K_{X})>n-1$ then $(-K_{X})^{n}>(n-1)^{n}$ and we
have $d=k+l+1>kl$. It follows that $k=1$ or $k=l=2$. The latter
case does not occur since we then have $d=5$ and $m=4$, hence $\epsilon(-K_{X})\le n-1$
by (\ref{eq:4}). Thus $k=1$ and $d=l+2$. If $l=2$, then $X$ is
the double cover of $\mathbb{P}^{n}$ branched over a general quartic
hypersurface. In particular $X$ is smooth and $\epsilon(-K_X)\le n-1$ when $n\ge 3$ by \cite[Theorem 1.1]{bcw}.
Hence we may assume $l\ge3$ and thus $d=l+2<2l$. Finally if $r\le0$ then $d\le k+l<2l$ unless $k=l$.
We conclude that if $\epsilon(-K_{X})>n-1$ and $n\ge3$ then we always
have $d<2l$ or $k=l=\frac{1}{2}d$. In both cases, by writing down the defining equation $f$ of $X$ it is not hard to see that $X$ is rational (if $d<2l$ then $f$ is linear on $x_{n+1}$ and if $k=l=\frac{1}{2}d$ then after a change of variables $f$ contains the term $x_nx_{n+1}$).
\end{expl}

\begin{expl}
Let $X=X_{6}\subseteq\mathbb{P}(1^{n-1},2^{2},3)=Y$ be a general
weighted hypersurface of degree 6. Then $-K_{X}\sim nH$ and $\epsilon(-K_{X})=\frac{2}{3}n$.
Indeed let $x\in X$ be a very general point and let $V\cong\mathbb{P}(1,2,2,3)\subseteq Y$
be the corresponding linear subspace containing $x$. Let $S=X\cap V$,
then $S$ is a Gorenstein log del Pezzo surface of degree 2 and $-K_{S}\sim2H|_{S}$.
By \cite[Th\'eor\`eme  1.3]{br06} we have $\epsilon(-K_{S})=\frac{4}{3}$, hence $\epsilon(H|_{S})=\frac{2}{3}<1$.
Similar to the previous example, the Seshadri constant of $H$ is the same
as its restriction on $S$, thus $\epsilon(-K_{X})=\frac{2}{3}n$
as claimed. Now if $\epsilon(-K_{X})>n-1$ then $n\le2$ and $X$
is again rational.
\end{expl}

Finally we may ask for the moving Seshadri version of Theorem \ref{main:n-1}.

\begin{que} \label{que:bdd}
Is the set of varieties $X$ of dimension $n$ with $\epsilon_m(-K_X)\ge n-1$ birationally bounded?
\end{que}

It is easy to see that the assumption about moving Seshadri constants here is optimal:

\begin{expl}
Let $C$ be a curve of genus $g$ and $D$ a divisor of degree $d\gg0$ on $C$. Consider $X=\bP_C(\cO_C\oplus\cO_C(-D))$. Let $F$ be a fiber of the natural projection $\pi:X\rightarrow C$ and $E$ the unique negative section. We have the Zariski decomposition \[-K_X=(1+\frac{2g-2}{d})E+(1-\frac{2g-2}{d})(E+dF)\]
which implies $\epsilon_m(-K_X)=(1-\frac{2g-2}{d})\epsilon(E+dF)=1-\frac{2g-2}{d}$. Varying $g$ an $d$ we see that for any $\epsilon>0$ the set of surfaces $X$ with $\epsilon_m(-K_X)>1-\epsilon$ is not birationally bounded. It is not hard to generalize this to $n$-dimensional examples by considering $\bP^{n-1}$-bundles over $C$.
\end{expl}


\bibliography{ref}
\bibliographystyle{alpha}

\end{document}